\documentclass[11pt]{article}

\usepackage{amsmath}         
\usepackage{amssymb}  
\usepackage{amsthm} 
\usepackage{enumerate}
\usepackage{stmaryrd}
\usepackage{url} 
\usepackage[all]{xy}
\usepackage{a4wide}            

\usepackage{esvect}      

\usepackage{mathrsfs}       

\usepackage{scalerel}   

\usepackage{enumerate}  

\usepackage{hyperref}

\numberwithin{equation}{section}

\theoremstyle{plain}
\newtheorem{theorem}{Theorem}[]
\newtheorem{lemma}[theorem]{Lemma}

\newtheorem{corollary}[theorem]{Corollary}

\theoremstyle{definition}
\newtheorem{definition}[theorem]{Definition}

\newtheorem{remark}[theorem]{Remark}
\newtheorem{example}[theorem]{Example}

 \DeclareMathOperator{\im}{Im}
  \DeclareMathOperator{\id}{id}

\newcommand{\bb}[1]{\mathbb{#1}}
   \newcommand{\A}{\bb{A}}
   \newcommand{\B}{\bb{B}}
   \newcommand{\LL}{\bb{L}}  
   \newcommand{\M}{\bb{M}}    
   \newcommand{\C}{\bb{C}} 
   
 \newcommand{\f}{\mathbf{f}}     
  \newcommand{\g}{\mathbf{g}} 
  
  \newcommand{\olat}{\Om\textbf{-Lat}}
   \newcommand{\onlo}{\Om\textbf{-NLO}}
  \newcommand{\opol}{\Om\textbf{-Pol}}

\newcommand{\cS}{\mathcal{S}}  

\newcommand{\CC}{\mathscr{C}}
\newcommand{\sL}{\mathscr{L}}
\newcommand{\V}{\mathscr{V}}
\newcommand{\Z}{\mathscr{Z}}

\renewcommand{\wp}{\mathscr{P}}

   \newcommand{\fF}{\mathscr{F}}
      \newcommand{\fI}{\mathscr{I}}

\newcommand{\al}{\alpha}
\newcommand{\be}{\beta}
\newcommand{\ep}{\varepsilon}

\newcommand{\lam}{\lambda}
\newcommand{\Lam}{\Lambda}
\newcommand{\om}{\omega}
\newcommand{\Om}{\Omega}
\newcommand{\ph}{\varphi}
\newcommand{\sg}{\sigma}     
 
\newcommand{\thet}{\theta} 
\newcommand{\Ups}{\Upsilon}     

\def\bo{\mathop{\scalerel*{\Box}{X}}\kern-1pt}
\def\di{\mathord{\scalerel*{\Diamond}{gX}}}

\newcommand{\ab}[1]{|#1|}

\newcommand{\du}{\partial}
\newcommand{\epi}{\twoheadrightarrow}

\renewcommand{\ge}{\succcurlyeq} 
\renewcommand{\geq}{\geqslant} 
\newcommand{\lap}{\between}      
\newcommand{\inv}{^{-1}}  
\newcommand{\join}{\bigvee}
\renewcommand{\le}{\preccurlyeq}  
\renewcommand{\leq}{\leqslant}  

\newcommand{\lo}{^{\triangledown}}               
\newcommand{\meet}{\bigwedge}
\newcommand{\mono}{\rightarrowtail}
\newcommand{\ov}[1]{\overline{#1}}
        
\newcommand{\sub}{\subseteq}  

\newcommand{\up}{^{\vartriangle}}

\newcommand{\vc}[2]{#1_0,\ldots{},#1_{#2-1}}

\begin{document}

\title{Morphisms and Duality for Polarities\\ and Lattices with Operators}
\author{Robert Goldblatt
\\
Victoria University of Wellington}
\date{}
\maketitle

\begin{abstract}   
Structures based on polarities have been used to provide relational semantics for propositional logics that are modelled algebraically by non-distributive lattices with additional operators. This article develops a first order notion of  morphism  between polarity-based structures that generalises the theory of bounded morphisms for Boolean modal logics.
It defines a category of such structures that is contravariantly dual to a given category of lattice-based algebras whose additional operations preserve either finite joins or finite meets.
Two different versions of the Goldblatt-Thomason theorem are derived in this setting.
\end{abstract}

\section{Introduction and overview}  

\hfill
\begin{minipage}{.5\textwidth} \small
Duality in mathematics is not a theorem, but a ``principle''.

\hfill Michael Atiyah
\end{minipage}

\bigskip

We  develop here   a new definition of `bounded morphism' between certain structures that  model propositional logics lacking the distributive law for conjunction and disjunction. Our theory adapts a well known semantic analysis of modal logic, which we now review.

 There are two main types of semantical interpretation of propositional modal logics.
In \emph{algebraic} semantics, formulas of the modal language are interpreted as elements of
a modal algebra $(\B,f)$, which comprises a Boolean algebra $\B$ with an additional operation $f$ that  interprets the
modality $\Diamond$ and preserves finite joins. In \emph{relational} semantics, formulas are interpreted as subsets of a Kripke frame $(X,R)$,  which comprises a binary relation $R$ on a set $X$.

The relationship between these two approaches is explained by a \emph{duality} that exists between algebras and frames. This is fundamentally category-theoretic in nature. The modal algebras are the objects of a category \textbf{MA} whose arrows are the standard algebraic homomorphisms. The Kripke frames are the objects of a category \textbf{KF} whose arrows are the  \emph{bounded morphisms},  $\al\colon (X,R)\to(X',R')$, i.e.\ functions $\al\colon X\to X'$ satisfying the `back and forth' conditions
\begin{align}
&\text{(Forth): \quad $xRy$ implies $\al(x)R'\al(y)$}.  \label{forth}
\\
&\text{(Back): \quad $\al(x)R'z$ implies $\exists y(xRy\ \&\ \al(y)=z)$}.  \label{back}
\end{align}
(Bounded morphisms are also known as p-morphisms. The adjective `bounded' derives from the $R$-bounded existential quantification in \eqref{back}.)

Each Kripke frame $\mathcal F=(X,R)$ has the dual modal algebra $\mathcal F^+=(\wp X,f_R)$ comprising the Boolean algebra of all  subsets of $X$, with the additional operation $f_R$ defined for all $A\sub X$ by
$$
f_R A=\{x\in X:\exists y(xRy\ \&\ y\in A)\}.
$$
Each modal algebra $\A=(\B,f)$ has the dual frame $\A_+=(X_\B,R_f)$, where $X_\B$ is the set of ultrafilters of $\B$, and
$xR_f y$ iff $\{f(a):a\in y\}\sub x$. There is an injective homomorphism of $\A$ into $(\A_+)^+$ that acts by $a\mapsto\{x\in X_\B:a\in x\}$, extending the Stone representation of Boolean algebras to modal algebras.
The mappings
$\mathcal F\mapsto\mathcal F^+$ and  $\A\mapsto\A_+$ form the basis of a pair of functors, from $\textbf{KF}$ to $\textbf{MA}$ and from
 $\textbf{MA}$ to $\textbf{KF}$ respectively, that are contravariant, i.e.\ direction-reversing. Each homomorphism $\thet\colon\A\to\A'$ induces a dual bounded morphism $\thet_+\colon\A'_+\to\A_+$, while each  bounded morphism
 $\al\colon \mathcal F\to \mathcal F'$ induces a dual homorphism   $\al^+\colon (\mathcal F')^+\to \mathcal F^+$. These induced maps act by forming preimages under $\thet$ and $\al$ respectively.

Composing the two functors produces objects of logical significance.
The double dual algebra $(\A_+)^+$ is known as the \emph{canonical extension} of $\A$, a construction first introduced by  J{\'o}nsson and Tarski \cite{jons:bool51} for Boolean algebras with any number of \emph{operators: }finitary operations that preserve  joins in each coordinate. They proved that any completely join preserving $n$-ary operator $f$ on $(\A_+)^+$ is determined by an $n+1$-ary relation $S_f$ on the structure $\A_+$, and showed that many equationally definable properties of  $f$ correspond to first-order definable properties of  $S_f$. This correpondence between algebras and relational structures provides tools for the semantic analysis of a range of  logics with modalities. One key to this is that if $\A$ is the Lindenbaum-Tarski algebra for a modal logic, then $\A_+$ is isomorphic to the \emph{canonical frame} for the logic whose points are maximally consistent sets of formulas \cite[Theorem 5.42]{blac:moda01}.

The study of canonical extensions has now evolved well beyond the Boolean situation. Gehrke and J\'onsson extended it to distributive lattice expansions \cite{gehr:boun94,gehr:mono00,gehr:boun04}. Then Gehrke and Harding \cite{gehr:boun01} gave an abstract algebraic definition of the canonical extension of any algebra $\A$ that is based on a bounded lattice. They proved the uniqueness of this extension up to isomorphism and constructed it as an algebra of `stable' subsets of a \emph{polarity},  a structure $(X,Y,R)$ comprising a binary relation $R$ between two sets $X$ and $Y$. Thus the polarity becomes the dual structure $\A_+$ of $\A$, and the canonical extension of $\A$ is the double dual $(\A_+)^+$,  which is the stable set lattice of $\A_+$.  Polarities are called  \emph{(formal) contexts} in Formal Concept Analysis \cite{gant:form99}.
The term `polarity' is itself of geometric origin, as we explain in Remark \ref{polgen} below.

Polarities with additional relational structure to represent additional algebraic operations have been used  by Gehrke and co-workers to  provide  relational semantics for several  logical systems, including the logic of residuated lattices \cite{dunn:cano05,gehr:gene06}, the  Lambek-Grishin calculus \cite{cher:gene12} and linear logic \cite{coum:rela14}, with canonical extensions playing a central role. There have also been applications to logics with unary modalities
 \cite{conr:cate16,conr:algo16}. 
 A fuller overview of the history of canonical extensions is given in the introduction to \cite{gold:cano18}.

Our objective here is to develop a new kind of morphism from a polarity $P=(X,Y,R)$ to a polarity $P'=(X',Y,'R')$  that can accomodate  expansion of the polarities by additional relational structure, and which  provides a dual morphism $\thet_+$ for any homomorphism $\thet$ of lattices with operators. 
Like \eqref{forth} and \eqref{back},
the definition of morphism we will use is \emph{first order}  relative to the structures involved, i.e.\ it  quantifies only over elements of the structures  and not over any higher order entities like subsets or sets of subsets.

The literature already contains several proposals for a notion of morphism between polarities.
Ern\'e \cite{erne:cate05,erne:tens94} investigated context morphisms as pairs of functions of the form $\al\colon X\to X'$ and  $\be\colon Y\to Y'$, and constructed functors between some categories of complete lattices with complete homomorphisms and categories of contexts with morphism-pairs having various properties. One of these, \emph{concept continuity}, is equivalent to our notion of morphism for polarities without additional structure, as we explain in Remark \ref{conccty}.
Hartung \cite{hart:topo92} studied mapping pairs $\al,\be$ between contexts with topological structure, and used them to obtain duals  for \emph{surjective} homomorphisms $\thet$. In \cite{hart:exte93} he obtained duals of arbitrary lattice homomorphisms by taking a morphism to be a pair of `multivalued functions', binary relations  forming subsets of $X\times X'$ and $Y\times Y'$. 
Hartonas and Dunn \cite{hart:ston97} defined morphisms as certain pairs of continuous functions between polarities with additional topological and partially ordered structure that characterises them as the duals of lattices (see  Remark \ref{conccty}).
There has also been work on polarity morphisms as pairs of subsets of  $X\times Y'$ and $Y\times X'$ that provide duals of completely join preserving homomorphisms \cite{dunn:cano05,gehr:gene06,conr:gold18}.
Jipsen \cite{jips:cate12} discusses a notion due to M.~A.~Moshier of a context morphism as a subset of $X\times Y'$, for which $R$ itself is the identity morphism on $P$.
Gehrke and van Gool \cite{gehr:dist14} studied polarity morphisms as pairs of functions satisfying back and forth properties similar to the modal  frame conditions \eqref{forth} and \eqref{back},  showing that they give duals for  lattice homomorphisms that preserve finite sets whose join distributes over meets, and ones whose  meet distributes over joins.

Here we define a bounded morphisms between polarity structures to be a pair $\al,\be$ of functions that have back and forth properties that look different to conditions \eqref{forth} and \eqref{back}, and in fact are similar to what would result from  those conditions if the relations $R$ and $R'$ were replaced by their \emph{complements}. For instance we use the reflection (back) condition
$$
\text{$\al  (x)R'\be  (y)$ implies $xRy$},
$$
in place of the preservation (forth) condition \eqref{forth}. The motivation for this approach comes from  earlier work of the author \cite{gold:sema74} in transforming polarity-style models of orthologic into  Kripke models of modal logic by replacing the polarity relation by its complement. Thus, at least for `ortho-polarities',  the bounded morphisms we use are  essentially equivalent to the modal bounded morphisms of their  transforms. This is explained in more detail in
Remark \ref{explainmorph}.

The new notion of bounded morphism allows us to carry out the kind of programme that was sketched above for the categories \textbf{MA} of modal algebras and \textbf{KF} of Kripke frames.  We  construct contravariant functors between a category $\olat$ of homomorphisms between lattices with additional operators (and dual operators) and a category $\opol$ of bounded morphisms between polarities with additional $n+1$-ary relations corresponding to  additional $n$-ary lattice operations. Bounded morphisms also gives rise to a notion of $P$ being an \emph{inner substructure} of $P'$, meaning that $P$ is a substructure of $P'$ for which the inclusions $X\hookrightarrow X'$ and $Y\hookrightarrow Y'$ form a bounded morphism. It is shown that  the image of a bounded morphism is an inner substructure of its codomain (Corollary \ref{imageinner}). Moreover,
the dual of a surjective homomorphism is a bounded morphism whose domain is isomorphic to its image (Theorem \ref{morphdual}).

On the other side of the duality to $(\A_+)^+$ is the double dual $(\mathcal F^+)_+$ of a frame $\mathcal F$, which we also call the \emph{canonical extension} of $\mathcal F$. It plays a central role in a definability result from \cite{gold:axio75}, generally known as the Goldblatt-Thomason theorem, which addresses the question of when a class of frames is definable by modal formulas. Here we consider the corresponding question for a class $\cS$ of polarity-based structures and show in Theorem \ref{GT} that if  $\cS$ is closed under canonical extensions, then it is equal to the class $\{P:P^+\in\V\}$ of all structures whose dual algebras belong to some equationally definable class of algebras $\V$ if, and only if, $\cS$ reflects canonical extensions and  is closed under images of bounded morphisms, inner substructures and direct sums.
The direct sum construction, introduced for polarities by  Wille \cite{will:rest82,will:subd87},  performs the same function here that disjoint unions perform for Kripke frames, namely it is dual to the formation of direct products of stable set lattices. We note that it is also a coproduct in the category $\opol$ that we define.

The original definability theorem from  \cite{gold:axio75} was concerned with modal definability of \emph{first-order} definable classes of frames, and its proof used the fact that any frame $\mathcal F$ has an elementary extension $\mathcal F^*$ that can be mapped surjectively onto $(\mathcal F^+)_+$ by a bounded morphism. This $\mathcal F^*$ can be taken to be an ultrapower of $\mathcal F$, so the theorem's hypothesis can be taken to be that $\cS$ is closed under ultrapowers. Along with closure under images of bounded morphisms this then yields the required closure under canonical extensions. Here we adapt the construction to polarities and find that there is a divergence from the modal case: the bounded morphism
$P^*\to ( P^+)_+$ may not have the surjectivity required for this argument.
But it does have a weaker property that allows a modified proof that  $\cS$ is closed under canonical extensions. 
We call this property \emph{maximal covering}
(briefly: the points of the $X$-part of $( P^+)_+$ are the filters of $P^+$ and the image of a maximal covering  morphism includes any filter that is maximally disjoint from some ideal).
Thus we obtain a  different definability characterisation (Theorem \ref{GT2}) for a class $\cS$ that is closed under ultrapowers, in which closure under codomains of maximal covering  morphisms replaces closure under images of bounded morphisms. An example is provided to show that this change is essential.

At the end of the article we briefly indicate how the theory can be extended to \emph{quasi}-operators, functions that in each coordinate either preserve joins or change meets into joins.

\section{Polarities and stable set lattices}    

We assume that all lattices dealt with have universal bounds, and view them as algebras of the form $(\LL,\land,\lor,0,1)$, with binary operations of meet $\land$ and join $\lor$, least element 0 and greatest element 1. 
The partial order of a lattice is denoted  $\leq$, and the symbols $\join$ and $\meet$ are used for the join and meet of a set of elements, when these exist. If they exist for all subsets, the lattice is \emph{complete}.

A \emph{polarity} is a structure $P=(X,Y,R)$ having $R\sub X\times Y$. For $A\sub X$ and $B\sub Y$, write $ARB$ if $xR y$ holds for all $x\in A$ and  $y\in B$. Abbreviate $AR\{y\}$ to $ARy$ and $\{x\}RB$ to $xRB$. Define
$$
\rho_R A=\{y\in Y:  ARy\},\quad 
\lam_R B=\{x\in X: xRB\}.
$$
The operations $\rho_R$ and $\lam_R$ are inclusion reversing: $A\sub A'$ implies $\rho_R A'\sub\rho_R A$, and likewise for $\lam_R$. They satisfy the `De Morgan laws'
\begin{equation}  \label{demorg}
\textstyle
\rho_R\bigcup\mathcal C=\bigcap\{\rho_R A:A\in\mathcal C\},\qquad
\lam_R\bigcup\mathcal C=\bigcap\{\lam_R B:B\in\mathcal C\},
\end{equation}
but not the corresponding laws with $\bigcup$ and $\bigcap$ interchanged.

The composite operations $\lam_R\rho_R$ on $\wp(X)$ and $\rho_R\lam_R$ on $\wp(Y)$ are closure operations whose fixed points are called \emph{stable} sets.
Thus a subset $A$ of $X$ is \emph{stable} if $A=\lam_R\rho_R A$, and a subset $B$ of $Y$ is \emph{stable} if $B=\rho_R\lam_RA$. In general  $\lam_R\rho_R A$ is the smallest stable superset of $A$ and $ \rho_R\lam_RB$ is the smallest stable superset of $B$, so to prove stability of $A$ it is enough to prove $\lam_R\rho_R A\sub A$, and similarly for $B$. The stable subsets of $X$ are precisely the sets of the form $\lam_R B$, and the stable subsets of $Y$ are precisely the sets of the form $\rho_R A$. This uses that under composition,  $\lam_R\rho_R\lam_R=\lam_R$ and $\rho_R\lam_R\rho_R=\rho_R$.

Let $P^+$ be the set of all stable subsets of $X$ in $P$, partially ordered by set inclusion. $P^+$ forms a complete lattice in which $\meet \mathcal C=\bigcap \mathcal C$,
$\join \mathcal C=\lam_R\rho_R\bigcup\mathcal C$, $1=X$ and $0=\lam_R\rho_R \emptyset=\lam_R Y$.
We call $P^+$ the \emph{stable set lattice} of $P$.
For any $A\in P^+$, we have

\begin{equation}\label{joinmeetdense}
A=\join\nolimits_{x\in A}\lam_R\rho_R\{x\}=\bigcap\nolimits_{y\in\rho_R A}\lam_R\{y\}.
\end{equation}

A quasi-order $\le_1$ on $X$, with inverse $\ge_1$, is defined by putting
\begin{equation} \label{leone}
x\le_1 x' \quad\text{iff}\quad \rho_R\{x\}\sub  \rho_R\{x'\}.
\end{equation}
Similarly, a quasi-order $\le_2$ on $Y$ is given by
\begin{equation}   \label{letwo}
y\le_2 y' \quad\text{iff}\quad \lam_R\{y\}\sub  \lam_R\{y'\}.
\end{equation}
Then the following condition holds:
\begin{equation}  \label{monR}
x'\ge_1 xR y\le_2 y' \text{ implies } x'Ry' .
\end{equation}
For $x\in X$ and $y\in Y$ put 
$$
[x)_1=\{x'\in X:x\le_1 x'\}, \qquad  [y)_2=\{y'\in Y:y\le_2 y'\}.
$$
A subset $A$ of $X$ is an \emph{upset} under $\le_1$ if it is closed upwards under $\le_1$,
i.e.\ $x\in A$ implies $[x)_1\sub A$. Likewise a set $B\sub Y$ is a \emph{$\le_2$-upset} if $y\in B$ implies $[y)_2\sub B$.

\begin{lemma}  \label{upsets}
Any stable subset of $X$ is a $\le_1$-upset, and any stable subset of $Y$ is a $\le_2$-upset.
Hence $\rho_RA$ is a $\le_2$-upset for any $A\sub X$, and $\lam_RB$ is a $\le_1$-upset for any $B\sub Y$.
\end{lemma}
\begin{proof}
Let $A\in P^+$. If $x\in A$ and $x\le_1 x'$, then any $y\in \rho_R A$ has 
$x'\ge_1 xR y\le_2 y$, hence $x' Ry$ by \eqref{monR}. So $x'\in\lam_R\rho_R A=A$. This shows $A$ is a $\le_1$-upset. 
The case of stable subsets of $Y$ is similar.
The second statement of the lemma follows as $\rho_RA$ and $\lam_RB$ are always stable.
\end{proof}

A map $\al:(X,\le)\to(X',\le')$ between  quasi-ordered sets is \emph{isotone} if it preserves the orderings, i.e.\ $x\le z$ implies $\al(x)\le' \al(z)$. For such a map, if $A$ is an $\le'$-upset  of $X'$, then $\al\inv A$ is an $\le$-upset  of $X$.

An  \emph{antitone} $\al$ is one that reverses the orderings, i.e.\ $x\le z$ implies $\al(z)\le' \al(x)$. For example, $\rho_R$ is antitone as a map $(\wp(X),\sub)\to(\wp(Y),{\sub})$. Likewise for $\lam_R\colon(\wp(Y),\sub)\to(\wp(X), \sub)$.

\begin{remark}[\bf Etymology of `polarity']        \label{polgen}
In projective plane geometry, a polarity is an interchange of points and lines, with the line associated to a given point being the \emph{polar} of the point, and  the point associated to a given line being the \emph{pole} of the line. A point $x$ lies on a given line iff the pole of that line lies on the polar of $x$.
The pole of the polar of a point is that point, and the polar of the pole of a line is that line.
Two points are called  \emph{conjugate} if each lies on the polar of the other. The polar of point $x$ can be identified with the set of points $\{y:xRy\}$ where $R$ is the congugacy relation.

A polarity on a projective three-space interchanges points and planes as poles and polars, while associating lines with each other in pairs. Two associated lines are polars of each other.
More generally, a polarity on a finite-dimensional projective space is an inclusion reversing permutation $\thet$ of the subspaces that is also an involution, i.e.\ $\thet(\thet A)=A$. Such a $\thet$ can be obtained from an inner product (symmetric bilinear function) $x\cdot y$ on the space by putting $\thet A=\rho_R A$, where $xRy$ iff $x\cdot y=0$.

The  use of `polarity' to refer to a binary relation derives from the work of Birkhoff \cite[Section 32]{birk:latt40} who first defined the operations $\rho_R$ and $\lam_R$ for an arbitrary $R\sub X\times Y$ and observed that they give a dual isomorphism between the lattices of stable subsets of $X$ and $Y$. He suggested that $\rho_R A$ could in general be called the \emph{polar} of $A$ with respect to $R$, in view of the above geometric example.
\qed
\end{remark}

\section{Operators and relations}

A finitary operation $f\colon \LL^n\to\LL$ on a lattice is an \emph{operator} if it preserves binary joins in each coordinate. 
As such it preserves the ordering of $\LL$ is each coordinate, which implies that it preserves the product ordering, i.e.\ it is isotone as an operation on $\LL^n$.

A \emph{normal operator} preserves the least element in each coordinate as well, hence preserves all  finite joins in each coordinate, including the empty join 0. A \emph{complete operator} preserves all existing non-empty joins in each coordinate, while a \emph{complete normal operator} preserves the empty join as well. By iterating the join preservation in each coordinate successively, one can show that if $f$ is a  complete normal operator, then
\begin{equation}\textstyle   \label{joincomplete}
f(\join A_0,\dots,\join A_{n-1})=\join\{f(\vc{a}{n}):a_i\in A_i \text{ for all }i<n\}.
\end{equation}
A \emph{dual operator} (\emph{normal} dual operator, \emph{complete} dual operator, \emph{complete normal} dual operator) is a finitary operation that preserves binary meets (finite meets, non-empty meets, all meets) in each coordinate. (Preservation of the empty meet means preservation of the greatest element $1$.) A dual operator is isotone on $\LL^n$.
A complete normal dual operator satisfies the equation that results from \eqref{joincomplete} by replacing each $\join$ by 
$\meet$.

Fix a polarity $P=(X,Y,R)$.
We are going to show that complete normal $n$-ary operators on the stable set lattice $P^+$ can be built from  $n+1$-ary relations on $P$. For this we need to introduce some vectorial notation for handling  tuples and relations (sets of tuples).               

A tuple $(\vc{x}{n})$ will be denoted $\vv{x}$. Then $\vv{x}[z/i]$ denotes the tuple obtained from $\vv{x}$ by replacing $x_i$ by $z$, while $(\vv{x},y)$ denotes the $n+1$-tuple $(\vc{x}{n},y)$. If $S\sub X^n\times Y$ is an $n+1$-ary relation, i.e.\  a set of $n+1$-tuples, we usually write  $\vv{x}Sy$  when $(\vv{x},y)\in S$. For $Z\sub X^n$ we write $ZSy$ if $\vv{x}Sy$ holds for all $\vv{x}\in Z$.

If $\vv{A}=(\vc{A}{n})$ is a tuple of sets $A_i$, we write $\pi\vv{A}$ for the product set
$A_0\times\cdots\times A_{n-1}$. We sometimes write $\vv{x}\in_\pi\vv{A}$ when $\vv{x}\in\pi\vv{A}$, i.e.\ when $x_i\in A_i$ for all $i<n$. 
Similarly  $\vv A\sub_\pi\vv B$ means that $A_i\sub B_i$ for all $i<n$.
Various operations are lifted to tuples coordinate-wise, so that $\rho_R\vv{A}=(\vc{\rho_R A}{n})$ while
$\lam_R\vv{A}=(\vc{\lam_R A}{n})$, $\thet\inv\vv{A}=(\vc{\thet\inv A}{n})$, etc.

A \emph{section} of a relation $S\sub X^n\times Y$ is any subset of $X$ or $Y$ obtained by fixing all but one of the coordinates and letting the unfixed coordinate vary arbitrarily. Thus  each $\vv{x}\in X^n$  determines the section
$$
S[\vv{x},-] = \{y\in Y:\vv{x}Sy\}.
$$
For $i<n$, sections that vary the $i$-th coordinate  are defined, for $\vv{x}\in X^n$ and $y\in Y$, by letting
$$
S[\vv{x}[-]_i,y]= \{x'\in X:\vv{x}[x'/i]Sy\}.
$$
We illustrate this definition with a technical lemma that will be applied below.
For each $x\in X$, let $\ab{x}=\lam_R\rho_R\{x\}$, the smallest member of $P^+$ to contain $x$. For an $n$-tuple $\vv{x}=(\vc{x}{n})$, let $\ab{\vv{x}}=(\ab{x_0},\dots,\ab{x_{n-1}})$. Note that since $x_i\in\ab{x_i}$ for all  $i<n$,  we have 
$\vv x\in_\pi\ab{\vv x}$.

\begin{lemma}  \label{lemvvxSy}
Suppose that all sections of $S$ of the form $S[\vv{x}[-]_i,y]$ are stable in $X$. Then for all $(\vv x,y)\in X^n\times Y$,  if $\vv xSy$  then every $\vv z\in_\pi\ab{\vv x}$ has $\vv zSy$, i.e.\ $(\pi\ab{\vv x})Sy$.
\end{lemma}

\begin{proof}
Let $\vv xSy$ where $\vv x=(\vc{x}{n})$. We will show by  induction on $i\leq n$ that
\begin{equation}  \label{indi}
\text{ if $z_j\in\ab{x_j}$ for all $j< i$, then
$(z_0,\dots,z_{i-1},x_i,\dots,x_{n-1})Sy$. }
\end{equation}
Putting $i=n$ then gives the desired result that
$(\ab{x_0}\times\cdots\times\ab{x_{n-1}})Sy$.

If $i=0$, then \eqref{indi} holds by the assumption $\vv xSy$. Now suppose inductively that \eqref{indi} holds for some $i<n$, and that $z_j\in\ab{x_j}$ for all $j< i+1$. Then this  hypothesis gives $\vv wSy$, where
$\vv w=(z_0,\dots,z_{i-1},x_i,\dots,x_{n-1})$.
Now   $S[\vv{w}[-]_i,y]$ is a stable set containing $x_i$, because $\vv w[x_i/i]=\vv w$ and $\vv w Sy$. Since $\ab{x_i}$ is the smallest such stable set, we get $\ab{x_i}\sub S[\vv{w}[-]_i,y]$. But $z_i\in\ab{x_i}$, so this implies $\vv w[z_i/i]Sy$, i.e.\ 
 $
 (z_0,\dots,z_{i},x_{i+1},\dots,x_{n-1})Sy.
 $
Hence \eqref{indi} holds with $i+1$ in place of $i$,
That completes the inductive proof that  \eqref{indi} holds for all $i\leq n$, as required. 
\end{proof}

Now for  $S\sub X^n\times Y$,  define  $f^\bullet_S\colon(P^+)^n\to\wp Y$ by putting, for  $\vv{A}\in(P^+)^n$,
\begin{align}
f^\bullet_S\vv{A}
&=\{y\in Y:  (\pi\vv{A})Sy\}, \nonumber
\\
&=\bigcap\{S[\vv{x},-]:\vv{x}\in\pi\vv{A}\}.  \label{f0stable}
\end{align}
Then define an $n$-ary operation $f_S$ on 
$P^+$ by putting 
$$
f_S\vv{A}=\lam_R f^\bullet_S\vv{A}.
$$
This definition generalises the form of the binary fusion operation $\otimes$ defined in \cite{gehr:gene06} from a  relation $S\sub X^2\times Y$ by 
\begin{align*}
A_0\otimes A_1&=\bigcap\{\lam_R\{y\}: (\forall x_0\in A_0)(\forall x_1 \in A_1)\, S(x_0,x_1,y)\}\\
&=\lam_R\{y\in Y:(A_0\times A_1)Sy\}.
\end{align*}

Note that $f^\bullet_S$ is antitone in the $i$-th coordinate, i.e.\ if $A_i\sub B$ then $f^\bullet_S(\vv{A})\supseteq f^\bullet_S(\vv{A}[B/i])$. Hence $f_S$ is isotone in each coordinate. The condition for $x\in f_S\vv{A}$ is
\begin{equation}  \label{firstfsA}
\forall y\in Y[ \forall \vv{z}(\vv{z}\in_\pi\vv{A}\to \vv{z}S y)\to xRy],
\end{equation}
which can be spelt out as a first-order formula in the predicates $z_i\in A_i$, $ \vv{z}S y$ and $xRy$.

\begin{theorem}
Let $f$ be any $n$-ary complete normal operator on $P^+$ for a polarity $P$. Then $f$ is equal to the operation $f_{S_f}$ determined by some relation $S_f\sub X^n\times Y$.
\end{theorem}
\begin{proof}
Recall that for  $\vv{x}=(\vc{x}{n})$ we put  $\ab{\vv{x}}=(\ab{x_0},\dots,\ab{x_{n-1}})$ where
 $\ab{x_i}=\lam_R\rho_R\{x_i\}\in P^+$.
Define $\vv{x}S_f y$ iff $y\in\rho_Rf\ab{\vv{x}}$. Then for $\vv{A}\in (P^+)^n$,
\begin{align}
f^\bullet_{S_f}\vv{A}
&=\{y\in Y:\vv{x}\in_\pi\vv{A} \text{ implies } y\in\rho_Rf\ab{\vv{x}}\} \label{deff0S}
\\
&=\bigcap\{\rho_Rf\ab{\vv{x}}: \vv{x}\in_\pi\vv{A}\}.  \label{meetrhoR}
\end{align}
But since $f$ preserves  joins in each coordinate, using the first equation from \eqref{joinmeetdense} we get
\begin{align*}
f\vv{A}
&=f(\join\nolimits_{x_0\in A_0}\ab{x_0},\dots, \join\nolimits_{x_{n-1}\in A_{n-1}}\ab{x_{n-1}} )
\\
&=\join\{f(\ab{\vv{x}}):\vv{x}\in_\pi\vv{A}\}  \qquad\text{by \eqref{joincomplete}},
\\
&=\lam_R\rho_R\big(\bigcup\{f(\ab{\vv{x}}):\vv{x}\in_\pi\vv{A}\} \big), \quad\text{by definition of }\join,
\\
&=\lam_R\big(\bigcap\{\rho_Rf(\ab{\vv{x}}):\vv{x}\in_\pi\vv{A}\} \big)  \quad \text{by }\eqref{demorg},
\\
&=\lam_R f^\bullet_{S_f}\vv{A}  \qquad\text{by \eqref{meetrhoR} },
\\
&=f_{S_f}\vv{A}.
\end{align*}
\end{proof}

\begin{theorem} \label{fScomplop}
If all sections of $S$ are stable, then $f_S$ is a complete normal operator, and S is equal to the relation $S_{f_S}$ determined by $f_S$.
\end{theorem}
\begin{proof}
Assume all sections of $S$ are stable.
To prove that $f_S$ preserves joins in the $i$-th coordinate it is enough to prove that the inclusion
\begin{equation} \textstyle \label{partialpres}
f_S(\vv{A}[\join_JB_j/i] )\sub \join_J f_S(\vv{A}[B_j/i])
\end{equation}
holds for any $\vv{A}\in(P^+)^n$ and any collection $\{B_j\in J\}\sub P^+$ with join $\join_JB_j$. This is because the converse  inclusion must hold, since $f_S$ is isotone in the $i$-th coordinate, so 
$f_S(\vv{A}[B_j/i])\sub f_S(\vv{A}[\join_JB_j/i] )$ for all $j\in J$. We will first show that

\begin{equation} \textstyle  \label{f0case}
\bigcap_{j\in J} f^\bullet_S(\vv{A}[B_j/i])\sub f^\bullet_S(\vv{A}[\join_JB_j/i] ).
\end{equation}
To see this, let $y\in f^\bullet_S(\vv{A}[B_j/i])$ for all $j\in J$. Take any $\vv{x}\in_\pi\vv{A}[X/i]$. Then for each $j$, if $x'\in B_j$ then $\vv{x}[x'/i]\in \pi\vv{A}[B_j/i]$, so $\vv{x}[x'/i]S y$ by the definition \eqref{deff0S} of $f^\bullet_S$. This shows that $B_j\sub S[\vv{x}[-]_i,y]$. But the latter section is a stable subset of $X$, hence belongs to $P^+$, so this implies that $\join_JB_j\sub S[\vv{x}[-]_i,y]$. Hence $\vv{x}[z/i]Sy$ for all $z\in\join_JB_j$. As that holds for all $\vv{x}\in_\pi\vv{A}[X/i]$, we get $\big(\pi\vv{A}[\join_JB_j/i] \big)Sy$, hence
   	$y\in f^\bullet_S(\vv{A}[\join_JB_j/i])$, proving \eqref{f0case}.

Now as $f^\bullet_S(\vv{A}[\join_JB_j/i]) $ is stable, being an intersection of stable sections \eqref{f0stable},
we reason that
\begin{align*} 
\bigcap\nolimits_{j\in J} f^\bullet_S(\vv{A}[B_j/i])
&=\bigcap\nolimits_{j\in J} \rho_R\lam_Rf^\bullet_S(\vv{A}[B_j/i])
\\
&=\bigcap\nolimits_{j\in J} \rho_Rf_S(\vv{A}[B_j/i])
\\
&= \rho_R\big(\bigcup\nolimits_{j\in J}f_S(\vv{A}[B_j/i])\big),
\end{align*}
and therefore from \eqref{f0case} that
$$
\rho_R\big(\bigcup\nolimits_{j\in J}f_S(\vv{A}[B_j/i])\big)
\sub f^\bullet_S(\vv{A}[\join\nolimits_JB_j/i] ).
$$
Hence $ \lam_R f^\bullet_S(\vv{A}[\join_JB_j/i] )\sub
\lam_R \rho_R\big(\bigcup\nolimits_{j\in J}f_S(\vv{A}[B_j/i])\big)$,
which is \eqref{partialpres}.

To show that $S=S_{f_S}$, note first that $\vv x S_{f_S}y$ iff $y\in\rho_R f_S\ab{\vv x}$, by definition of $S_f$, while 
$\rho_R f_S\ab{\vv x}=\rho_R\lam_Rf_S^\bullet\ab{\vv x}=f_S^\bullet\ab{\vv x}$ since $f_S^\bullet\ab{\vv x}$ is stable by \eqref{f0stable} as all $S$-sections are stable. Thus    $\vv x S_{f_S}y$ iff  $y\in f_S^\bullet\ab{\vv x}$ iff
$(\pi\ab{\vv x})Sy$.
But if  $(\pi\ab{\vv x})Sy$, then $\vv xSy$ since $\vv x\in\pi\ab{\vv x}$. Conversely, if $\vv xSy$ then $(\pi\ab{\vv x})Sy$ by Lemma \ref{lemvvxSy}.  So altogether, $\vv x S_{f_S}y$  iff $\vv xSy$.
\end{proof}
An alternative proof that $f_S$ is a complete normal operator can be given by showing that the function
$A'\mapsto  f_S(\vv{A}[A'/i] )$ has a right adjoint, namely the function
$B'\mapsto  g^i(\vv{A}[B/i] )$,
where
$$g^i\vv{A}=\bigcap\{S[\vv{x}[-]_i,y]:\vv{x}\in\vv{A}[X/i] \text{ and } y\in\rho_R A_i\}.
$$
The adjointness means that 
$f_S(\vv{A}[A'/i] )\sub B$ iff $A'\sub g^i(\vv{A}[B/i] )$ for any $A',B\in P^+$. It is a standard fact that any lattice operation with a right adjoin preserves all joins.
The functions $B'\mapsto  g^i(\vv{A}[B/i] )$ for each $i<n$ generalise the two residual operations 
$ A_0\backslash A_1$ and $A_0/A_1$
of the above binary fusion operation $A_0\otimes A_1$, as given in \cite{gehr:gene06}. These can be expressed as
\begin{align*}
 A_0\backslash A_1 &= \bigcap\{S[x_0,-,y]:  x_0\in A_0\ \&\ y\in \rho_R A_1 \},
 \\
 A_0/ A_1                 &=\bigcap\{S[-,x_1,y]:  x_1\in A_1\ \&\   y\in \rho_R A_0 \}.
\end{align*} 

Next we consider the construction of dual operators on $P^+$.
An $m$-ary dual operator can be obtained from a relation of the form $T\sub X\times Y^m$. We write $xT\vv{y}$ when $(x,\vv{y})\in T$,  and $xTZ$ when $xT\vv{y}$ holds for all $\vv{y}\in Z$. Sections of $T$ take the form
\begin{align*}
T[-,\vv{y}] &=\{x\in X:xT\vv{y}\} \qquad\qquad \text{ for }\vv{y}\in Y^m,
\\
T[x,\vv{y}[-]_i] &=\{y'\in Y:xT\vv{y}[y'/i]\}  \qquad \text{for }x\in X,\ i<m,\ \vv{y}\in Y^m.
\end{align*}
Dually to Lemma \ref{lemvvxSy} we have a result about sequences of the form 
$\ab{\vv{y}}=(\ab{y_0},\dots,\ab{y_{m-1}})$, where $\ab{y_i}=\rho_R\lam_R\{y_i\}$, the smallest stable subset of $Y$ to contain $y_i$. Suppose all sections of $T$ of the form $T[x,\vv{y}[-]_i] ]$ are stable. Then if $xT\vv y$, we can show by induction on $i\leq m$ that
if $z_j\in\ab{y_j}$ for all $j< i$, then
$xT(z_0,\dots,z_{i-1},y_i,\dots,y_{m-1})$.
Putting $i=m$ then gives

\begin{lemma}  \label{lemxTvvy}
Suppose that all sections of $T$ of the form $T[x,\vv{y}[-]_i] ]$ are stable in $Y$. Then for all $(x,\vv y)\in X\times Y^m$,  if $ xT\vv y$  then every $\vv z\in_\pi\ab{\vv y}$ has $xT\vv z$, i.e.\ $xT (\pi\ab{\vv y})$. 
\qed
\end{lemma}

We now define an $m$-ary function $g_T$ on $P^+$.  
For $\vv A\in(P^+)^m$ put 
\begin{align}
g_T\vv{A}
&=\{x\in X: xT( \pi\rho_R\vv{A}) \}, \nonumber
\\
&= \{x\in X:  \forall\vv{y}(\vv{y}\in\pi\rho_R\vv{A}\text{ implies }xT\vv y)\} \nonumber
\\
&=\bigcap\{T[-,\vv{y}]:\vv{y}\in_\pi\rho_R\vv{A}\}.  \label{GTA}
\end{align}
The condition for $x\in g_T\vv A$ can be spelt out as the first-order expression
\begin{equation}\textstyle  \label{firstgT}
\forall \vv y\big(\bigwedge_{i<m}\forall z(z\in A_i\to zRy_i)\to xT\vv y\big).
\end{equation}

We exemplify $g_T$ with the case that $m=1$. Then $T$ is a binary relation from $X$ to $Y$, inducing a unary operation on $P^+$ which we denote  more suggestively by $\bo_T$. Thus
$$
\bo_T A= \{x\in X: xT\rho_R A\}.
$$
As a binary relation, $T$ can also be viewed as  a subset of $X^n\times Y$ with $n=1$, so it induces a unary operation  
$\di_T$ on $P^+$  having
$$
\di_T A= \lam_R\{y\in Y: ATy\}.
$$
When all sections of $T$ are stable,  $\bo_T$ is a dual operator and $\di_T$ is an operator that
 is left adjoint to $\bo_T$ in the sense that for any $A,B\sub X$,
$$
\di _TA\sub B \quad\text{iff}\quad A\sub\bo_T B.
$$
See  \cite{conr:cate16,conr:algo16} for discussion of pairs of `modalities' like $\bo_T$ and $\di_T$.

\begin{theorem} \label{gTcompletedual}
If all sections of a relation $T\sub X\times Y^m$  are stable, then $g_T$ is a complete normal dual operator.
\end{theorem}
\begin{proof}
If all sections of $T$ are stable then $g_T\vv{A}$ is stable by \eqref{GTA}, so $g_T$ is an operation on $P^+$.
$g_T$ is isotone in each coordinate. Hence it satisfies the inclusion
$$
\textstyle
g_T(\vv{A}[\,\bigcap\nolimits_JB_j/i] )\sub \bigcap\nolimits_J g_T(\vv{A}[B_j/i]).
$$
To show that $g_T$ is a complete normal dual operator, we prove that the last inclusion is an equality for any $i<m$. Let $x\in\bigcap\nolimits_{j\in J} g_T(\vv{A}[B_j/i])$, where $\{B_j:j\in J\}\sub P^+$. Then
\begin{equation} \label{allJrhoBj}
\forall j\in J\ \forall \vv{y}\in_\pi\rho_R(\vv{A}[B_j/i]),\ xT\vv{y}.
\end{equation}
Now take any $\vv{y}\in_\pi\rho_R(\vv{A}[\bigcap_JB_j/i])$. Then for any $j\in J$, if $y'\in\rho_R B_j$ then
$\vv{y}[y'/i]\in_\pi\rho_R(\vv{A}[B_j/i])$, so $xT(\vv{y}[y'/i])$ by \eqref{allJrhoBj}. This shows that 
$
\rho_RB_j\sub T[x,\vv{y}[-]_i].
$
Therefore $\lam_R T[x,\vv{y}[-]_i]  \sub\lam_R \rho_RB_j=B_j$. Hence 
$\lam_R T[x,\vv{y}[-]_i]\sub\bigcap_JB_j$, giving
$$\textstyle
\rho_R\bigcap\nolimits_JB_j\sub \rho_R\lam_R T[x,\vv{y}[-]_i] =T[x,\vv{y}[-]_i]
$$ as    $T[x,\vv{y}[-]_i]$ is stable.
But $y_i\in\rho_R\bigcap_JB_j$, so then $y_i\in T[x,\vv{y}[-]_i]$, making $xT\vv{y}$.
Altogether we have shown that
$$\textstyle
\vv{y}\in_\pi\rho_R(\vv{A}[\,\bigcap\nolimits_JB_j/i] )\text{ implies }xT\vv{y},
$$
which means that $x\in g_T(\vv{A}[\,\bigcap\nolimits_JB_j/i] )$, completing the proof that $g_T$ preserves all meets in the $i$-th coordinate.
\end{proof}

It can also be shown that if $g$ is any complete normal dual operator on $P^+$, then $g$ is equal to  $g_{T_g}$, where 
$T_g\sub X\times Y^m$ is defined by
$$
xT_g\vv{y} \quad\text{iff}\quad x\in g(\lam_R\{y_0\},\dots,\lam_R\{y_{m-1}\}),
$$
making $T_g[-,\vv{y}]=g(\lam_R\{y_0\},\dots,\lam_R\{y_{m-1}\})$.
Then using the second equation from \eqref{joinmeetdense}, for any $\vv A\in(P^+)^m$ we get
\begin{align*}
g\vv{A}  
&= g\Big(\bigcap\nolimits_{y_0\in \rho_RA_0}\lam_R\{y_0\},\dots, \bigcap\nolimits_{y_{m-1}\in \rho_RA_{m-1}}\lam_R\{y_{m-1}\}\Big)
\\
&=\bigcap\{g(\lam_R\{y_0\},\dots,\lam_R\{y_{m-1}\}):\vv{y}\in_\pi\rho_R\vv{A}\} \quad\text{by the dual of \eqref{joincomplete}},
\\
&=\bigcap\{T_g[-,\vv{y}]:\vv{y}\in_\pi\rho_R\vv{A}\}
\\
&=g_{T_g}\vv{A} \qquad\text{by }\eqref{GTA}.
\end{align*}
Also, if all sections of $T\sub X\times Y^m$ are stable, then $T$ is equal to the relation $T_{g_T}$ determined by $g_T$.
For, if $xT_{g_T}\vv y$ then by definition of $T_{g_T}$, 
$x\in g_T\vv{A}$ where $\vv{A}=(\lam_R\{y_0\},\dots,\lam_R\{y_{m-1}\})$, hence by definition of $g_T$, we get
$xT(\pi\rho_R\vv{A}) $. But $\vv y\in \pi\rho_R\vv{A}$ since $y_i\in\rho_R\lam_R\{y_i\}$ for $i<m$, so this implies $xT\vv y$. Conversely, if  $xT\vv y$, then by Lemma \ref{lemxTvvy}, $xT (\pi\ab{\vv y})$.  But here $\pi\ab{\vv y}=\pi\rho_R\vv{A}$, so this gives $x\in g_T\vv A$ and hence  $xT_{g_T}\vv y$ .
Altogether, $T=T_{g_T}$.

We are going to work with lattices having additional operators and dual operators, and we need a convenient notation for this. Let  $\Om$ be a set of function symbols with given finite arities. Define an \emph{$\Om$-lattice} to be
an algebra of the form
$$
\LL=(\LL_0,\{\f^\LL:\f\in\Om\}),
$$
where $\LL_0$ is a  lattice,  and for $n$-ary $\f\in\Om$, $\f^\LL$ is an $n$-ary operation on  $\LL_0$. 
Furthermore, we will take $\Om$ to be presented as the union $\Lam\cup\Ups$ of disjoint subsets $\Lam$ and $\Ups$  of  `lower' and `upper' symbols, respectively  (the reason for these names will emerge later---see \eqref{Lsigom}).
An $\Om$-lattice will be called a \emph{normal lattice with operators}---an $\Om$-NLO or just NLO---if each lower symbol denotes a normal operator in $\LL$ and each upper symbol denotes a normal dual operator.

We define an \emph{$\Om$-polarity} to be a structure of the form 
$$
P=(X,Y,R,\{S_\f:\f\in\Lam\},\{T_\g:\g\in\Ups\}),
$$
based on a polarity $P_0=(X,Y,R)$, such that for any $n$-ary lower symbol $\f\in\Lam$, $S_\f\sub X^n\times Y$ and all sections of $S_\f$ are stable;  and for any $m$-ary upper symbol $\g\in\Ups$, $T_\g\sub X\times Y^m$ and all sections of $T_\g$ are stable. Then $P$ gives rise to the $\Om$-NLO
\begin{equation}   \label{PplusOm}
P^+=(P_0^+, \{f_{S_\f}:\f\in\Lam\},\{g_{T_\g}:\g\in\Ups\}),
\end{equation}
where $f_{S_\f}$ is the complete normal operator determined by $S_\f$, as per  Theorem \ref{fScomplop}, 
and $g_{T_\g}$ is the complete normal dual operator determined by $T_\g$,  as per  Theorem \ref{gTcompletedual}.

\section{Bounded morphisms} \label{secbmorph}

To simplify the exposition, we fix two arbitrary natural numbers $n$ and $m$ and assume from now that $\Lam$ consists of a single $n$-ary function symbol while $\Ups$ consists of a single $m$-ary one.
Then an $\Om$-polarity has the typical form
$$
P=(X,Y,R,S,T)
$$
with $S\sub X^n\times Y$, $Y\sub X\times Y^m$, and all sections of $S$ and $T$ being stable. We lift the relations $\le_1$ and $\le_2$ to tuples coordinate-wise, putting $\vv x\le_1\vv{z}$ iff $x_i\le_1 z_i$ for all $i<n$; with
$[\vv x)_1=\{\vv{z}\in X^n:\vv x\le_1\vv z\}$ etc.

Let $P$ and $P'$ be $\Om$-polarities of the kind just described. For a function $\al:X\to X'$ we put
\begin{align*}
\al(\vv x)&=(\al(x_0),\dots,\al(x_{n-1})), 
\\
\al \inv[\vv{x'})_1&=\{\vv x\in X^n:\vv{x'}\le_1'\al(\vv x)\} \enspace\text{etc.}
\end{align*}

\begin{definition}  \label{defmorph}
A \emph{bounded morphism from $P$ to $P'$} is a pair $\al,\be$ of \emph{isotone}
maps $\al\colon(X,{\le_1)}\to (X',\le'_1)$ and  $\be:(Y,\le_2)\to (Y',\le_2')$
that satisfy the following back and forth conditions.

\begin{enumerate}
\item[(1$_R$)]
$\al  (x)R'\be  (y)$ implies $xRy$, \quad all $x\in X, y\in Y$.
\item[(2$_R$)]
$(\al \inv[x')_1)Ry$ implies $x'R'\be (y)$, \quad all $x'\in X',y\in Y$.
\item[(3$_R$)]
$xR\be \inv[y')_2$ implies $\al (x)R'y'$, \quad all $x\in X,y'\in Y'$.
\item[(1$_S$)]
$\al(\vv x)S'\be  (y)$ implies $\vv xSy$, \quad all $\vv x\in X^n,\  y\in Y$.
\item[(2$_S$)]
$(\al \inv[\vv{x'})_1)Sy$ implies $\vv{x'}S'\be (y)$, \quad all $\vv{x'}\in (X')^n,\ y\in Y$.
\item[(1$_T$)]
$\al  (x)T'\be  (\vv y)$ implies $xT\vv y$, \quad all $x\in X,\ \vv y\in Y^m$.
\item[(2$_T$)]
$xT\be \inv[\vv{y'})_2$ implies $\al (x)T'\vv{y'}$, \quad all $x\in X,\ \vv{y'}\in (Y')^m$.
\end{enumerate}
\end{definition}

In condition (3$_R$), $\be \inv[y')_2$ is the set $\{y\in Y:y'\le_2'\be(y)\}$. Thus the condition can be expressed as
$$
\text{ if $\forall y(y'\le_2'\be(y)\text{ implies }xRy)$, then $\al(x)R' y'$}.
$$
Contrapositively this says
\begin{equation} \label{pmorph}
\text{ if not $\al(x)R'y'$, then $\exists y(y'\le_2'\be(y)\text{ and not }xRy)$}.
\end{equation}
Similar formulations hold for (2$_R$), (2$_S$) and (2$_T$). Note that the converse of \eqref{pmorph} is equivalent to (1$_R$), which follows from this converse by putting $y'=\be(y)$.  To derive the converse, observe that if not $xRy$ then 
(1$_R$) implies not $\al(x)R'\be(y)$, so then if $y'\le_2'\be(y)$ we get not $\al(x)R'y'$ by definition of $\le_2'$.

\begin{remark} [\bf Source of Definition 8]\label{explainmorph}
Suppose that $X=Y$ and $R$ is symmetric. Then we have the kind of polarity used in \cite{gold:sema74} to provide a semantics for orthologic, in which the operation $\lam_R$ ($=\rho_R$) is an orthocomplementation modelling a negation connective. An ortholattice representation was given in \cite{gold:ston75} in which the points of the representing space are filters of the lattice, and which can be seen as a precursor to the canonical structures of Section \ref{seccanstr} below.
A translation of orthologic  into classical modal logic was obtained in \cite{gold:sema74} by transforming an `orthoframe' $(X,R)$ into the Kripke frame $(X,\widehat{R})$, where $\widehat{R}=X^2\setminus R$ is the complementary relation to $R$. Taking $\al=\be$, so that (2$_R$) and (3$_R$) become equivalent, then
if the conditions (1$_R$)--(3$_R$) are re-expressed in terms of $\widehat{R}$ and $\widehat{R'}$,   they become similar to the standard definition of a bounded morphism  between the Kripke frames $(X,\hat R)$ and $(X',\widehat{R'})$, with (1$_R$) being equivalent to \eqref{forth} for $\widehat{R}$, and   (3$_R$) in the form \eqref{pmorph} amounting to \eqref{back} for $\widehat{R}$ (except for the relation $\le'_2$).  The other conditions in Definition \ref{defmorph} give  parallel back and forth properties for the relations $S$ and $T$.
This account may explain why it is often natural to use contrapositive reasoning in proofs of properties of bounded morphisms, as we shall see.

The use of quasi-orderings $\le_i$ is well established in theories of duality for non-Boolean lattices \cite{prie:repr70,urqu:topo78} and relates to the fact that  the points of dual structures are typically filters and/or ideals that may not be maximal. Such points are naturally quasi-ordered by the set inclusion relation $\sub$. 
The use of a quasi-order to formulate bounded morphism conditions like  \eqref{pmorph} is also well established \cite[p.192]{gold:vari89}, \cite[p.698]{cela:prie99}.
We could adopt a more axiomatic approach and let $\le_1$ and $\le_2$ be any additional quasi-orders that satisfy the condition \eqref{monR}, which is equivalent to requiring only that
$x\le_1 x' $ implies  $\rho_R\{x\}\sub  \rho_R\{x'\}$ and
$y\le_2 y'$ implies $ \lam_R\{y\}\sub  \lam_R\{y'\}$.
But in a polarity, suitable quasi-orders can be defined as in \eqref{leone} and \eqref{letwo}, and shown to give the relation $\sub$ in a canonical structure: see  Lemma \ref{lesub}. 

The requirement that $\al$ and $\be$ be isotone is needed to ensure that the class of bounded morphisms is closed under functional composition and gives a category: see Lemma \ref{compmorph}.
\qed
\end{remark}

We will show that a  bounded morphism makes the following diagrams commute, where the operation $\wp_{\le}$ gives the set of all $\le$-upsets.

$$
\xymatrix{
{\wp_{\le_1}X} \ar[r]^{\rho_{R} }  &{\wp_{\le_2}Y} 
\\
{\wp_{\le_1'}X'} \ar[u]^{\al\inv} \ar[r]^{\rho_{R'}} &{\wp_{\le_2'}Y'}   \ar[u]_{\be\inv}   }
\qquad\qquad
\xymatrix{
{\wp_{\le_1}X}   &\ar[l]_{\lam_{R} } {\wp_{\le_2}Y} 
\\
{\wp_{\le_1'}X'} \ar[u]^{\al\inv} &{\wp_{\le_2'}Y'} \ar[l]_{\lam_{R'}}   \ar[u]_{\be\inv}   }
$$

\begin{lemma} \label{morph}
Let $\al,\be\colon P\to P'$ be a bounded morphism and $A\sub X'$ and $B\sub Y'$.
\begin{enumerate}[\rm(1)]
\item 
$\be \inv(\rho_{R'} A)\sub\rho_R(\al \inv A)$, with 
$\be \inv(\rho_{R'} A)=\rho_R(\al \inv A)$ when $A$ is a $\le_1'$-upset.
\item
$\al \inv(\lam_{R'} B)\sub\lam_R(\be \inv B)$, with 
$\al \inv(\lam_{R'} B)=\lam_R(\be \inv B)$ when  $B$ a $\le_2'$-upset.
\item
If  $A\in (P')^+$ then $\al \inv A\in P^+$.
\end{enumerate}
\end{lemma}
\begin{proof}
\begin{enumerate}[\rm(1)]
\item 
Let $y\in \be \inv(\rho_{R'} A)$, so $AR'\be (y)$.
If $x\in \al\inv A$, then $\al(x)\in A$, so then $\al(x)R'\be (y)$, hence $xRy$ by (1$_R$). This shows $(\al\inv A)Ry$, making  
$y\in \rho_R(\al\inv A)$.

Suppose further that $A$ is a $\le_1'$-upset. Let $y\in \rho_R(\al\inv A)$, so $(\al\inv A)Ry$. If  $x'\in A$, then $[x')_1\sub A$ as $A$ is an upset, so $\al\inv[x')_1\sub \al\inv A$, hence $\al\inv[x')_1Ry$, and so $x'R'\be (y)$ by (2$_R$). This shows $AR'\be (y)$, so $\be (y)\in \rho_{R'} A$ and 
$y\in \be \inv(\rho_{R'} A)$.

\item
Like (1), but using  (1$_R$) and (3$_R$). 

\item
$A=\lam_{R'}(\rho_{R'} A)$ as $A$ is stable, so
$\al\inv A=\al\inv\lam_{R'}(\rho_{R'} A)
=\lam_R\be \inv(\rho_{R'} A)$ by part (2) as $\rho_{R'} A$ is stable, therefore a $\le'_2$-upset.
But any subset of $X$ of the form $\lam_R B$ is stable and so belongs to $P^+$.

\end{enumerate}
\end{proof}

\begin{corollary}   \label{equivmorph}
\begin{enumerate}[\rm(1)]
\item 
A pair $\al,\be$ satisfies  (1$_R$) and (2$_R$) if, and only if,
\begin{equation}  \label{equiv12R}
\text{$\be \inv(\rho_{R'} A)=\rho_R(\al \inv A)$ for all stable $A\sub X'$.}
\end{equation}
\item
$\al,\be$ satisfies  (1$_R$) and (3$_R$) if, and only if,
\begin{equation}  \label{equiv13R}
\text{$\al \inv(\lam_{R'} B)=\lam_R(\be \inv B)$ for all stable $B\sub Y'$.}
\end{equation}
\end{enumerate}
\end{corollary}

\begin{proof}
(1): If  (1$_R$) and (2$_R$) hold, then \eqref{equiv12R} follows by part (1) of the Lemma because stable sets are upsets.
Conversely, assume \eqref{equiv12R}. To prove (1$_R$): if $\al(x)R'\be(y)$, then
$$
y\in\be\inv\rho_{R'}\{\al(x)\}=\be\inv\rho_{R'}\lam_{R'}\rho_{R'}\{\al(x)\}=
\rho_R\al \inv \lam_{R'}\rho_{R'}\{\al(x)\},
$$
with the last equation holding by \eqref{equiv12R}. But $x\in \al \inv \lam_{R'}\rho_{R'}\{\al(x)\}$, so then $xRy$.

For  (2$_R$), suppose that not $x'R'\be(y)$. Then
$
y\notin\be\inv\rho_{R'}\{x'\} 
=\rho_R\al \inv \lam_{R'}\rho_{R'}\{x'\}
$
as above. Hence there exists $x\in X$ such that not $xRy$ and
$\al(x)\in \lam_{R'}\rho_{R'}\{x'\}$. Then
$\rho_{R'}\{x'\}=\rho_{R'}\lam_{R'}\rho_{R'}\{x'\}\sub \rho_{R'}\{\al(x)\}$, so $x'\le_1'\al(x)$ and 
$x\in \al \inv[x')_1$. Since not $xRy$ this gives not $(\al \inv[x')_1)Ry$ as required.

The proof of (2) is similar.
\end{proof}

\begin{remark} [\bf On conditions \eqref{equiv12R} and \eqref{equiv13R}] \label{conccty} 
Formal Concept Analysis defines a \emph{concept} of a context/polarity $P=(X,Y,R)$  to be a pair $(A,B)$ of subsets, of $X$ and $Y$ respectively, with $A=\lam_R B$ and $B=\rho_R A$. The set of concepts is partially ordered by putting
$(A,B)\leq(C,D)$ iff $A\sub C$ (iff $D\sub B$), forming a complete lattice isomorphic to $P^+$.
Ern\'e \cite{erne:cate05} defined a pair $X\xrightarrow{\al} X',Y\xrightarrow{\be} Y'$ to be \emph{concept continuous} if
$(\al\inv A,\be\inv B)$ is a concept of $P$ whenever $( A, B)$ is a concept of $P'$.
He showed in \cite[Prop.~3.2]{erne:cate05} that $\al,\be$ is concept continuous iff the follow conditions hold.
\begin{align} \label{pmorph1}
\text{not $x'R'\be(y)$} &\text{ iff there is an $x$ with not $xRy$ and } \rho_{R'}\{x'\}\sub \rho_{R'}\{\al(x)\}, 
\\
\label{pmorph2}
\text{not $\al(x)R'y'$} &\text{ iff there is a $y$ with not $xRy$ and }\lam_{R'}\{y'\}\sub \lam_{R'}\{\be(y)\}.
\end{align}
Now it is readily seen that $\al,\be$ is concept continuous iff conditions \eqref{equiv12R} and \eqref{equiv13R} of our Corollary \ref{equivmorph} hold,  which is equivalent by that corollary  to having  (1$_R$)--(3$_R$).
Condition \eqref{pmorph2} is equivalent to the combination of \eqref{pmorph} and its converse, which we already noted is equivalent to having  (1$_R$) and  (3$_R$). Likewise, \eqref{pmorph1} is equivalent to having  (1$_R$) and  (2$_R$).

Hartonas \cite{hart:ston18}, building on \cite{hart:ston97}, studied morphisms between polarity-based structures in which $\le_1$ and $\le_2$ are complete partial orders,  $X$ and $Y$ each carry a Stone space topology, the closed subsets of $X$ are the sets $[x)_1$ while the open subsets are the sets $\lam_R\{y\}$, and similarly for subsets of $Y$. A morphism
is  a pair $X\xrightarrow{\al} X',Y\xrightarrow{\be} Y'$ of continuous functions that preserve all meets in $X$ and $Y$ respectively, and are such that $\al\inv$ and $\be\inv$ satisfy the equations in \eqref{equiv12R} and \eqref{equiv13R} for all clopen stable $A$ and $B$.
\qed
\end{remark}

In the proof of the next result, and elsewhere, the notation $\al[Z]$ will be used for the \emph{image} $\{\al(z):z\in Z\}$ of a set $Z$ under function $\al$.

\begin{theorem} \label{homdual}
For any bounded morphism $\al,\be\colon P\to P'$,  the  map
 $A\mapsto \al\inv A$ gives a homomorphism $$(\al,\be)^+\colon (P')^+\to P^+$$ of\/ $\Om$-lattices.
 If $\al$ is surjective,  $(\al,\be)^+$ is injective.
  If $\al$ is injective,  $(\al,\be)^+$ is surjective.
\end{theorem}

\begin{proof}
The map is well defined by Lemma \ref{morph}(3). It preserves binary meets because
$\al\inv(A\cap B)=\al\inv A\cap \al\inv B$. It preserves binary joins because

$\al\inv(A\lor B)$

$=\al\inv\lam_{R'}\rho_{R'} (A\cup B)$ \quad by definition of $A\lor B$,

$=\lam_R\be\inv \rho_{R'} (A\cup B)$ \quad by Lemma \ref{morph}(2),

$=\lam_R\rho_R \al\inv  (A\cup B)$\quad  by Lemma \ref{morph}(1),

$=\lam_R\rho_R( \al\inv A\cup \al\inv  B)$ \quad by property of $\al\inv$,

$=\al\inv A\lor \al\inv  B$.

\noindent
It preserves  greatest elements as $\al\inv X'=X$, and   least elements as $\al\inv\lam_{R'}Y'=\lam_R \be\inv Y'=\lam_R Y$.

Next we show that $(\al,\be)^+$ preserves the operations $f_S$ and $f_{S'}$, first proving that for any $\vv A\in ((P')^+)^n$,
\begin{equation} \label{fcirchom}
\be\inv f^\bullet_{S'}\vv A =f^\bullet_S\al\inv\vv A.
\end{equation}
Let $y\in \be\inv f^\bullet_{S'}\vv A$, so $(\pi\vv A)S'\be(y)$. Then if $\vv x\in_\pi\al\inv\vv A$, then $\al(\vv x)\in\pi\vv A$, so
$\al(\vv x)S'\be(y)$, hence $\vv xSy$ by (1$_S$). This shows that $(\pi\al\inv\vv A)Sy$, so $y\in f^\bullet_S(\al\inv\vv A)$.

Conversely, let $y\in f^\bullet_S(\al\inv\vv A)$, so $(\pi\al\inv\vv A)Sy$. Take any $\vv{x'}\in\pi\vv A$.  Then if $\vv z\in\al\inv[\vv{x'})_1$, we have $\vv{x'}\le_1\al(\vv z)$, so $\al(\vv z)\in\pi\vv A$ as each $A_i$ is a $\le_1$-upset, hence
$\vv z\in\pi\al\inv\vv A$, and therefore $\vv zSy$.
Thus $(\al\inv[\vv{x'})_1) Sy$. Hence $\vv{x'}S'\be (y)$ by (2$_S$). Altogether this shows that $(\pi\vv A)S'\be(y)$. Therefore 
$\be(y)\in f^\bullet_{S'}\vv A$, hence $y\in \be\inv f^\bullet_{S'}\vv A$, which completes the proof of \eqref{fcirchom}. Now we reason that
\begin{align*}
f_S(\al\inv\vv A)
&=\lam_Rf^\bullet_S\al\inv\vv A
\\
&=\lam_R\be\inv f^\bullet_{S'}\vv A   \qquad \text{by \eqref{fcirchom}} 
\\
&=\al\inv\lam_{R'}f^\bullet_{S'}\vv A    \qquad \text{by Lemma \ref{morph}(2)}
\\
&= \al\inv f_{S'}\vv A,  
\end{align*}
proving that $(\al,\be)^+$ preserves the operations $f_S$ and $f_{S'}$.

Now we prove that $(\al,\be)^+$ preserves  $g_T$ and $g_{T'}$, i.e.\ for any $\vv A\in ((P')^+)^m$,
\begin{equation} \label{thetinvpresg}
\al\inv g_{T'}\vv A =g_T\al\inv\vv A.
\end{equation}
Let $x\in \al\inv g_{T'}\vv A$, so $\al(x)\in g_{T'}\vv A$, i.e.\ $\al(x)T'\pi\rho_{R'}\vv A$. Now if
$$
\vv y\in \pi\rho_{R}\al\inv\vv A =\pi\be\inv\rho_{R'}\vv A
$$
(see Lemma \ref{morph}(1)), then $\be(\vv y)\in\pi\rho_{R'}\vv A$,  hence $\al(x)T'\be(\vv y)$, and thus $xT\vv y$ by (1$_T$). This shows that $xT(\pi\rho_{R}\al\inv\vv A)$, i.e.\  $x\in g_T\al\inv\vv A$.

Conversely, assume $x\in g_T\al\inv\vv A$.   Let $\vv{y'}\in\pi\rho_{R'}\vv A$. Then if $\vv y\in\be\inv[\vv{y'})_2$, then $\vv{y'}\le_2\be(\vv y)$, so $\be(\vv y)\in\pi\rho_{R'}\vv A$. Hence for all $i<m$, 
$y_i\in\be\inv\rho_{R'}A_i=\rho_R\al\inv A_i$, and so  $\vv y\in\pi\rho_R\al\inv\vv A$. That gives $xT\vv y$ because
$x\in g_T\al\inv\vv A$. Altogether this show $xT \be\inv[\vv{y'})_2$. Hence $\al(x)T'\vv y'$ by (2$_T$). So we have established that  $\al(x)T'\pi\rho_{R'}\vv A$, which means that $\al(x)\in g_{T'}\vv A$.
Hence $x\in \al\inv g_{T'}\vv A$, completing the proof of \eqref{thetinvpresg}, and hence the proof that  $(\al,\be)^+$ is  an $\Om$-lattice  homomorphism.

The fact that  $A\mapsto \al\inv A$ is injective when $\al$ is surjective is standard set theory: surjectivity of $\al$ implies that $\al[\al\inv Z]=Z$ in general. Hence if $ \al\inv A= \al\inv B$, then $A= \al[\al\inv A]=\al[\al\inv B]=B$.

Finally,  suppose that $\al$ is injective. For any $B\in P^+$,  take $A=\lam_{R'}\rho_{R'}\al[B]$ $  \in (P')^+$. Then  by Lemma \ref{morph},
 $\al\inv A=\lam_{R}\rho_{R}\al\inv(\al[ B])=\lam_{R}\rho_{R}B$ as $\al$ is injective. But $B$ is stable, so we get that $B=\al\inv A=(\al,\be)^+A$, showing that $(\al,\be)^+$ is surjective. 
\end{proof}

Let $\opol$ be the category whose objects are the $\Om$-polarities and whose arrows are the bounded morphisms between such objects. The identity arrow $\id_P$ on each object $P$ is the pair $\id_X,\id_Y$ of identity functions on $X$ and $Y$. The composition of two arrows  
$$
\xymatrix{P \ar[r]^{\al,\be}    &P'   \ar[r]^{\al',\be'} &P'' }
$$
is given by the pair of functional compositions  $\al'\circ\al\colon X\to X''$ and  $\be'\circ\be\colon Y\to Y''$.

\begin{lemma}  \label{compmorph}
The pair $\al'\circ\al,\be'\circ\be$ is a bounded morphism from $P$ to $P''$. 
 \end{lemma}
 \begin{proof}
 It is straightforward that the composition of isotone functions is isotone.
 For the back and forth conditions we give the details for  (1$_S$) and (2$_S$), since the others follow the same pattern.
 
 For  (1$_S$), observe that if $(\al'\circ\al)(\vv x)S''(\be'\circ\be)(y)$, then $\al(\vv x)S'\be(y)$ by (1$_S$) for $\al',\be'$, hence $\vv xSy$ by  (1$_S$) for $\al,\be$.
 
For (2$_S$), we argue contrapositively and
take   $\vv{x''}\in (X'')^n$ and $y\in Y$ such that not $\vv{x''}S''(\be'\circ\be)(y)$.
Then as $\al',\be'$ is a bounded morphism, there exists $\vv{x'}\in (X')^n$ such that $\vv{x''}\le_1\al'(\vv{x'})$ and not $\vv{x'}S'\be(y)$.
Hence as $\al,\be$ is a bounded morphism, there exists $\vv{x}\in X^n$ such that $\vv{x'}\le_1\al(\vv{x})$ and not $\vv{x}Sy$. From $\vv{x'}\le_1\al(\vv{x})$ we get $\al'(\vv{x'})\le_1\al'(\al(\vv{x}))$ \emph{as $\al'$ is isotone}, hence its action on tuples is isotone. Since $\vv{x''}\le_1\al'(\vv{x'})$ it follows that  $\vv{x''}\le_1\al'(\al(\vv{x}))$, so we have $\vv x\in(\al'\circ\al)\inv[\vv{x''})_1$ while not $\vv xSy$, hence not $(\al'\circ\al)\inv[\vv{x''})_1Sy$,
confirming (2$_S$) for the pair  $\al'\circ\al,\be'\circ\be$.

The cases of  (3$_R$) and  (2$_T$) depend on $\be'$ being isotone.
\end{proof}

Let $\onlo$ be the category whose objects are the normal $\Om$-lattices with operators and whose arrows are the algebraic homomorphisms between $\Om$-NLO's, with the composition of arrows being their functional composition and the identity arrows being the identity functions. Then with the help of Theorem \ref{homdual} and Lemma \ref{compmorph} we see that the mappings $P\mapsto P^+$ and  $\al,\be\mapsto(\al,\be)^+$ form  a contravariant functor from $\opol$ to $\onlo$.

\section{Isomorphism and inner substructures}

Every category provides a definition of isomorphism between its objects. Thus the existence of $\opol$ allows us to read off a description of  isomorphisms between $\Om$-polarities. A bounded morphism $\mu\colon P\to P'$ is called an isomorphism if there exists a bounded morphism $\mu'\colon P'\to P$ such that $\mu'\circ\mu=\id_P$ and $\mu\circ\mu'=\id_{P'}$. Then $\mu'$ is the \emph{inverse} of $\mu$. It is itself an isomorphism, with inverse $\mu$.

We say that a bounded morphism $\al,\be$ \emph{preserves polarity}, or \emph{preserves} $R$,  if $xRy$ implies $\al(x)R'\be(y)$, i.e.\ if the converse of (1$_R$) holds.  Similarly  $\al,\be$ \emph{preserves $S$} if $\vv xSy$ implies $\al(\vv x)S'\be(y)$, and 
\emph{preserves $T$} if $xT\vv y$ implies $\al(x)T'\be(\vv y)$.

The function $\al$ \emph{reflects quasi-order} if  $\al(x)\le_1' \al(z)$ implies $x\le_1 z$. Likewise $\be$ \emph{reflects quasi-order} if $\be(y)\le'_2 \be(w)$ implies $y\le_2 w$.

\begin{lemma}  \label{prespol}
For any bounded morphism $\al,\be$ the following statements are equivalent.
\begin{enumerate}[\rm(1)]
\item 
$\al,\be$ preserves  polarity.
\item
$\al$ reflects quasi-order.
\item
$\be$ reflects quasi-order.
\end{enumerate}
\end{lemma}
\begin{proof}
(1) implies (2): Assume (1).
Let $x,z\in X$ have $\al(x)\le_1'\al(z)$. Then for all $y\in Y$, if $xRy$,  then $\al(x)R'\be(y)$ by (1),  and so $\al(z)R'\be(y)$ by definition of  $\al(x)\le_1'\al(z)$, hence $zRy$ by (1$_R$). Thus $\rho_{R}\{x\}\sub \rho_{R}\{z\}$, i.e.\ $x\le_1 z$.  Altogether this proves (2).

(2) implies (1): Assume (2). We prove (1) contrapositively. Suppose that not $\al(x)R'\be(y)$. Then by (2$_R$), there exists $z\in X$ with  $\al(x)\le_1'\al(z)$ and not $zRy$. Hence $x\le_1 z$ by (2), so by definition of $\le_1$ we get not $xRy$ as required for (1).

That completes a proof that (1) is equivalent to (2). Similarly we can prove that (1) is equivalent to (3), using (3$_R$) in place of (2$_R$).
\end{proof}

\begin{theorem}  \label{wheniso}
A bounded morphism $\mu=(\al,\be)\colon P\to P'$ is an isomorphism iff 
$\al\colon X\to X'$ and $\be\colon Y\to Y'$ are bijective
and preserve the relations $R$, $S$ and $T$.
\end{theorem}

\begin{proof}
Suppose there is a bounded morphism $\mu'=(\al',\be')$ that is inverse to $\mu$. Then $\al'\circ\al=\id_X$ and $\al\circ\al'=\id_Y$, so $\al$ has an  inverse $\al'$ and therefore is bijective. Similarly $\be$ is bijective.
Also, if $\vv xSy$,  then $\al'(\al(\vv x))S\be'(\be(y))$, so $\al(\vv x)S'\be(y)$ by (1$_S$) for the bounded morphism $\al',\be'$. This proves that $S$ is preserved. The proof that $R$ and $T$ are preserved is similar, using (1$_R$) and (1$_T$).

Conversely, assume that $\al$ and $\be$ are bijective and preserve the relations. Then   $\al$ and $\be$ have inverses
$\al'\colon X'\to X$ and $\be'\colon Y'\to Y$. We show that $\mu'=(\al',\be')$ is a bounded morphism, which will then be an inverse for $\mu$, as required for $\mu$ to be an isomorphism.

First, as $R$ is preserved, $\al$ and $\be$ reflect quasi-orders by Lemma \ref{prespol}.
If $x',z'\in X'$ have $x'\le_1'z'$, then there exist  $x,z\in X$ with $x'=\al(x)$ and $z'=\al(z)$, so $x\le_1 z$ from reflection by
 $\al$. But this means that $\al'(x')\le_1 \al'(z')$, and shows that $\al'$ is isotone. Similarly $\be'$ is isotone.

To prove the back condition (1$_S$) for  $\al',\be'$, let $\al'(\vv{x'})S\be'(y)$.
Then as  $\al,\be$ preserves $S$ we get $\al(\al'(\vv{x'}))S'\be(\be'(y))$, i.e.\ 
$\vv{x'}S'y$ as required. The cases of (1$_R$) and  (1$_T$) for  $\al',\be'$ are similar, using the preservation of $R$ and $T$.

For a forth condition for  $\al',\be'$, we prove (2$_S$). Let $\vv x\in X^n$ and $y'\in Y'$, and suppose that
not $\vv x S\be'(y')$. We have to show that not 
$(\al')^{-1}[\vv x)_1 S'y'$, i.e.\ that there exists $\vv w\in(X')^n$ such that $\vv x\le_1\al'(\vv w)$ and not $\vv wS'y'$.
Now by (1$_S$) for  $\al,\be$, from not $\vv x S\be'(y')$ we get that  not $\al(\vv x) S'\be(\be'(y'))$, hence by (2$_S$) for  $\al,\be$  we get  not $\al^{-1}[\al(\vv x))_1 S\be'(y')$. So there exists $\vv z\in \al^{-1}[\al(\vv x))_1 $, hence $\al(\vv x)\le'_1\al(\vv z)$, such that not $\vv zS\be'(y')$.
Let $\vv w=\al(\vv z)$. Then as $\al'$ is isotone, $\al'(\al(\vv x))\le_1\al'(\al(\vv z))$, i.e.\ $\vv x\le_1\al'(\vv w)$. Also from not 
$\vv zS\be'(y')$, by (1$_S$) for  $\al,\be$ we get not $\al(\vv z)S'\be(\be'(y'))$, i.e.\ not $\vv wS'y$ as required.

That proves (2$_S$) for  $\al',\be'$. The cases of the other forth condition for  $\al',\be'$ are similar.
\end{proof}

We define  $P$ to be an \emph{inner substructure} of $P'$ if the following holds.

\begin{enumerate}[\rm(i)]
\item 
$P$ is a substructure of $P'$ in the usual sense that  $X\sub X'$, $Y\sub Y'$, and the relations of $P$ are the restrictions of those of $P'$, i.e.\ 
$R=R'\cap (X\times Y)$,
$S=S'\cap (X^n\times Y)$, and
$T=T'\cap (X\times Y^m)$.
\item
The pair of inclusion maps $X\hookrightarrow X$ and $Y\hookrightarrow Y'$ form a bounded morphism from $P$ to $P'$. \end{enumerate}
Condition (i) here entails that  (1$_R$), (1$_S$) and (1$_T$) hold when $\al$ and $\be$ are the inclusions, e.g.\  (1$_R$) asserts that $xR'y$ implies $xRy$ when $x\in X$ and $y\in Y$.    Given (i), condition (ii) is equivalent to requiring the following.

\begin{enumerate}
\item[(2$_R$)]
$\{x\in X:x'\le_1' x\}Ry$ implies $x'R' y$, \quad all $x'\in X',y\in Y$;
\item[(3$_R$)]
$xR\{y\in Y:y'\le_2' y\}$ implies $xR'y'$, \quad all $x\in X,y'\in Y'$;
\item[(2$_S$)]
$\{\vv x\in X^n:\vv{x'}\le_1' \vv x\}Sy$ implies $\vv{x'}S'y$, \quad all $\vv{x'}\in (X')^n,\ y\in Y$;
\item[(2$_T$)]
$xT\{\vv y\in Y^m:\vv{y'}\le_2'\vv y\}$ implies $xT'\vv{y'}$, \quad all $x\in X,\ \vv{y'}\in (Y')^m$.
\end{enumerate}

\begin{theorem} \label{innersubdual}
It $P$ is an inner substructure of $P'$, then the map $A\mapsto A\cap X$ is an $\Om$-lattice  homomorphism from $(P')^+$ onto $P^+$.
\end{theorem}
\begin{proof}
If $\al$ is the inclusion $X\hookrightarrow X'$, then  $\al\inv A= A\cap X$, so by Theorem \ref{homdual}, 
 $A\mapsto A\cap X$ is a surjective homomorphism $ (P')^+\to P^+$ of  $\Om$-lattices, since $\al$ is injective.
 \end{proof}

The \emph{image}  of a bounded morphism $\al,\be:P\to P'$ is defined to be the structure
$$
\im(\al,\be)=(\al[X],\be[Y],R'',S'',T''),
$$
where the relations displayed are the restrictions of the corresponding relations of $P'$, i.e.\ 
$R''=R'\cap (\al[X]\times \be[Y])$,
$S''=S'\cap (\al[X]^n\times \be[Y])$, and
$T''=T'\cap (\al[X]\times \be[Y]^m)$.

\begin{lemma}  \label{Iminner}
\begin{enumerate}[\rm(1)]
\item
The quasi-orders $\le''_1$ and $\le''_2$ defined from $R''$ are the restrictions of the relations $\le'_1$ and $\le'_2$ to $\al[X]$ and $\be[Y]$, respectively.
\item 
All sections of $S''$ and $T''$ are stable in $\im(\al,\be)$.
\end{enumerate}
\end{lemma}
\begin{proof}
For part (1), we show that  $\al(x)\le''_1\al(z)$ iff $\al(x)\le'_1\al(z)$ for all $x,z\in X$.
Suppose first that $\al(x)\le'_1\al(z)$. By definition this means that $\al(x)R'y'$ implies $\al(z)R'y'$ for all $y'\in Y'$ (see \eqref{leone}).
In particular, for any $y\in Y$, if $\al(x)R''\al(y)$ then $\al(x)R'\al(y)$, hence $\al(z)R'\al(y)$ and so $\al(z)R''\al(y)$.
This shows $\rho_{R''}\{\al(x)\}\sub\rho_{R''}\{\al(z)\}$, i.e.\  $\al(x)\le''_1\al(z)$.

Conversely, let $\al(x)\le''_1\al(z)$. For any $y'\in Y'$, if not $\al(z)R'y'$, then by (3$_R$) there exists $y\in Y$ such that $y'\le'_2\be(y)$ and not $zRy$. Hence not $\al(z)R'\be(y)$ by (1$_R$), and so not $\al(z)R''\be(y)$. This implies
not $\al(x)R''\be(y)$ because $\al(x)\le''_1\al(z)$. Hence not $\al(x)R'\be(y)$, so then not $\al(x)R'y'$ as 
$y'\le'_2\be(y)$. Altogether this proves 
$\rho_{R'}\{\al(x)\}\sub\rho_{R'}\{\al(z)\}$, i.e.\  $\al(x)\le'_1\al(z)$.

The proof that $\le''_2$ is the restriction of $\le'_2$ to $\be[Y]$ is similar.

For part (2),
consider a section of the form $S''[\al(\vv{x})[-]_i,\be(y)]$.
If an element $\al(z)$ of $\al[X]$ does not belong to this section, then it does not belong to 
 $S'[\al(\vv{x})[-]_i,\be(y)]$. But the latter section is stable in $P'$, so there some 
 \begin{equation}  \label{yinrho}
 y'\in\rho_{R'}S'[\al(\vv{x})[-]_i,\be(y)]
 \end{equation}
 such that not $\al(z)R'y'$.
 Then  not $zR\be\inv[y')_2$ by (3$_R$), so
  there is some $w\in Y$ such that $y'\le'_2\be(w)$ and not $zRw$.
 Hence by (1$_R$), not $\al(z)R'\be(w)$, and therefore not $\al(z)R''\be(w)$.
 Now we show that
  \begin{equation}  \label{bewinrho}
 \be(w)\in\rho_{R''}S''[\al(\vv{x})[-]_i,\be(y)].
 \end{equation}
 For if  $t\in S''[\al(\vv{x})[-]_i,\be(y)]$, then  $t\in S'[\al(\vv{x})[-]_i,\be(y)]$, so by \eqref{yinrho}, $tR'y'$.
 But $y'\le'_2\be(w)$, so then $tR'\be(w)$ by \eqref{monR}.
 Hence  $tR''\be(w)$  as $t\in \al[X]$.
  This proves \eqref{bewinrho}. Since  not $\al(z)R''\be(w)$, $\al(z)\notin\lam_{R''}\rho_{R''}S''[\al(\vv{x})[-]_i,\be(y)]$, completing the proof that $S''[\al(\vv{x})[-]_i,\be(y)]$ is stable.

The argument for sections of the form $S''[\al(\vv{x}),-]$ is similar, using  (2$_R$). The arguments for the sections of $T''$ follow the same patterns. 
\end{proof}

\begin{corollary}  \label{imageinner}
$\im(\al,\be)$ an inner substructure of $P'$ .
\end{corollary}

\begin{proof}
Part (2) of the Lemma confirms that $\im(\al,\be)$ is an $\Om$-polarity.
By part (1), the inclusions  $(\al[X],\le''_1)\hookrightarrow (X',\le'_1)$   and
$(\be[Y],\le''_2)\hookrightarrow(Y',{\le'_2)}$ are isotone.   We  show that they satisfy the back and forth properties of Definition \ref{defmorph},  so they form a bounded morphism, making $\im(\al,\be)$ an inner substructure of $P'$ by definition. Since $\im(\al,\be)$ is defined to be a substructure of $P'$, the inclusions  do satisfy the back conditions, as already noted.

 For the forward conditions, we consider (2$_R$). This requires that for any $x'\in X'$ and $w\in\be[Y]$, if not $x'R'w$ then there exists $z\in\al[X]$ with $x'\le'_1 z$ and not $z R'' w$. Now $w=\be(y)$ for some $y\in Y$, and $\al$ and $\be$ satisfy 
(2$_R$), so if  not $x'R'\be(y)$ then there exists $x\in X$ with $x'\le'_1 \al(x)$ and not $x Ry$. Hence not $\al(x) R'\be(y)$
by (1$_R$), and so not $\al(x) R''\be(y)$. Thus putting $z=\al(x)$ fulfills our requirement for  (2$_R$). The proofs that the inclusions satisfy the other forward conditions are similar.
\end{proof}
Thus we have the general fact that the image of a bounded morphism is an inner substructure of its codomain.

\begin{theorem} \label{converseback}
If  $\al$ and $\be$ are injective and preserve the relations $R$, $S$ and $T$, then they give an isomorphism between $P$ and $\im(\al,\be)$. 
\end{theorem}

\begin{proof}
By Lemma \ref{Iminner}(1), as $\al$ is isotone as a map from
$(X,\le_1)$ to $(X',\le'_1)$, it is isotone as a map from
$(X,\le_1)$ to $(\al[X],\le''_1)$.
Likewise $\be$ is isotone as a map from
$(Y,\le_2)$ to $(\be[Y],\le''_2)$, so  $\al,\be$ acts as a bounded morphism from $P$ to $\im(\al,\be)$.

If  $\al$ and $\be$ are injective, then  they are bijective as maps to $\al[X]$ and $\be[Y]$, so if they
 preserve the relations as well then Theorem \ref{wheniso} ensures that they give an isomorphism from $P$ to $\im(\al,\be)$.
 \end{proof}

\section{Canonical extensions}

Lattice homomorphisms are assumed to preserve the  bounds 0 and 1, as well as $\land$ and $\lor$. An injective  homomorphism (\emph{mono}morphism) may be denoted by $\mono$, and a surjective  one (\emph{epi}morphism) by  
$\epi$. 
A function $\thet\colon\LL\to\M$ between lattices is called a \emph{lattice embedding}  if it is a lattice monomorphism. A lattice embedding is always  \emph{order invariant}, i.e.\ has 
$a\leq b$ iff $\thet a\leq \thet b$.

First we review the definition of a canonical extension of  a lattice, as given in \cite{gehr:boun01}.
A \emph{completion} of lattice $\LL$ is a pair $(\thet,\C)$ with $\C$ a complete lattice and $\thet\colon\LL\mono\C$  a lattice embedding.
An element of $\C$ is  \emph{open} if it is a join of elements from the $\thet$-image $\thet[\LL]$ of $\LL$ and \emph{closed} if it is a meet of elements from $\thet[\LL]$. Members of $\thet[\LL]$ are both open and closed. The set of open elements of the completion is denoted $O(\C)$, and the  set of closed elements is  $K(\C)$.

A completion $(\thet,\C)$ of $\LL$ is \emph{dense} if $K(\C)$ is join-dense and $O(\C)$ is meet-dense in $\C$, i.e.\ if  every member of $\C$ is both a join of closed elements and a meet of open elements. 
It is \emph{compact} if for any set $Z$ of closed elements and any set $W$ of open elements such that $\meet Z\leq\join W$,  there are finite sets $Z'\sub Z$ and $W'\sub W$ with $\meet Z'\leq\join W'$. An equivalent formulation of this condition that we will  use  (in Theorem \ref{thcano})  is that for any subsets $Z$ and $W$ of $\LL$ such that $\meet \thet[Z]\leq\join \thet[W]$ there are finite sets $Z'\sub Z$ and $W'\sub W$ with $\meet Z'\leq\join W'$.

A \emph{canonical extension} of  lattice $\LL$ is a completion $(\thet_\LL,\LL^\sg)$  of $\LL$ which is dense and compact. Any two such completions are isomorphic by a unique isomorphism commuting with the embeddings of $\LL$. This legitimises talk of ``the'' canonical extension, and the assignment of a name $\LL^\sg$ to it.

A function $f\colon\LL\to \M$ between lattices can be lifted it to a function $\LL^\sg\to\M^\sg$ between their canonical extensions in two ways, using the embeddings $\thet_\LL\colon \LL\mono\LL^\sg$ and  $\thet_\M\colon \M\mono\M^\sg$ to form the \emph{lower} canonical extension $f\lo$ and \emph{upper} canonical extension $f\up$ of $f$ (in \cite{gehr:boun01} these are denoted $f^\sg$ and $f^\pi$ respectively). 
Let $\mathbb{I}$  be the set of all intervals of the form $\{x:p\leq x\leq q\}$ in $\LL^\sg$ with $p\in K(\LL^\sg)$ and $q\in O(\LL^\sg)$.
Then for $x\in \LL^\sg$,
\begin{align}
f\lo x &=\join\big\{\meet \{\thet_\M(fa):a\in\LL\text{ and }\thet_\LL(a)\in E\}:x\in E\in\mathbb{I}\big\}. \label{floE}
\\
f\up x &=\meet\big\{\join\{\thet_\M(fa):a\in\LL\text{ and }\thet_\LL(a)\in E\}:x\in E\in\mathbb{I}\big\}.  \label{fupE}
\end{align}
The functions $f\lo$ and $f\up$  have $f\lo x\leq f\up x$. They both extend $f$ in the sense that the diagram
$$
\newdir{ >}{{}*!/-8pt/@{>}}
\xymatrix{
\LL  \ar@{ >->}[d]_{\thet_\LL} \ar[r]^f  &\M \ar@{ >->}[d]^{\thet_\M} 
\\
{\LL^\sg} \ar[r]^{h} &{\M^\sg}   }
$$
commutes when $h= f\lo$ or $h= f\up$.
For \emph{isotone} $f$,  these extensions can also  be specified as follows \cite[Lemma 4.3]{gehr:boun01}.
\begin{align}
f\lo p &= \meet \{\thet_\M(fa) :a\in\LL\text{ and }p\leq \thet_\LL(a)\},   \quad \text{ for all }  p\in K(\LL^\sg). \label{floclosed}
\\
f\lo x &=\join\{ f\lo p:p\in K(\LL^\sg) \text{ and }p\leq x\}, \quad \text{ for all } x\in\LL^\sg.
\\
f\up q &= \join \{\thet_\M(fa) :a\in\LL\text{ and }q\geq \thet_\LL(a)\}   \quad \text{ for all }  q\in O(\LL^\sg).\label{fupopen}
\\
f\up x &=\meet\{ f\up q:q\in O(\LL^\sg) \text{ and }q\geq x\}, \quad \text{ for all } x\in\LL^\sg.   \label{fup}
\end{align}

If $f\colon\LL^n\to\LL$ is an $n$-ary operation on $\LL$, then $f\lo$ and $f\up$ are maps from $(\LL^n)^\sg$ to $\LL^\sg$. But $(\LL^n)^\sg$ can be identified with $(\LL^\sg)^n$, since the natural embedding $\LL^n\rightarrowtail (\LL^\sg)^n$ is dense and compact, so this allows $f\lo$ and $f\up$ to be regarded as an $n$-ary operations on $\LL^\sg$.
Moreover it is readily seen that $K((\LL^\sg)^n)=( K(\LL^\sg))^n$, i.e.\ a closed element of $(\LL^\sg)^n$ is an $n$-tuple of closed elements of  $\LL^\sg$, and likewise $O((\LL^\sg)^n)=( O(\LL^\sg))^n$. This will be important below, where we apply the lower canonical extension to operators on $\LL$, and the upper extension to dual operators.

For any $\Om$-lattice $\LL$, based on a lattice $\LL_0$, we
define a canonical extension $\LL^\sg$ for $\LL$ by taking the canonical extension of $\LL_0$, applying the lower extension to operations denoted by members of $\Lam$, and 
the upper extension to operations denoted by members of $\Ups$, to form
\begin{equation}  \label{Lsigom}
\LL^\sg = (\LL_0^\sg, \{(\f^\LL)\lo:\f\in\Lam\}\cup\{(\g^\LL)\up:\g\in\Ups\}).   
\end{equation}
It is shown in \cite[Section 4]{gehr:boun01} (see also  \cite[2.2.14]{vosm:logi10}) that if $\LL$ is an NLO, then so is $\LL^\sg$, with each $(\f^\LL)\lo$ being a complete normal operator, and each $(\g^\LL)\up$ being a complete normal dual operator.

\section{Canonical structures}  \label{seccanstr}

The existence of $\LL^\sg$ was established in \cite{gehr:boun01} by taking  it to be the stable set lattice of a certain polarity between filters and ideals of $\LL$, with the additional operations of $\LL$ being extended to $\LL^\sg$ by  the abstract lattice-theoretic definitions \eqref{floE} and \eqref{fupE}. We will now see that if  $\LL$ is an $\Om$-NLO, then the polarity can be expanded to an $\Om$-polarity, which we call the \emph{canonical structure} of $\LL$, and whose stable set $\Om$-lattice, as in \eqref{PplusOm}, is a canonical extension of $\LL$.

Recall that we are assuming that $\Lam=\{\f\}$ and $\Ups=\{\g\}$. In what follows we  write the $n$-ary operator $\f^\LL$ just as $f$ and the $m$-ary dual operator $\g^\LL$ just as $g$. Let $\fF_\LL$ be the set of non-empty filters of $\LL$ and  $\fI_\LL$ be the set of non-empty ideals of $\LL$.
For $F\in\fF_\LL$ and $D\in\fI_\LL$, write $F\lap D$ to mean that $F$ and $D$ intersect, i.e.\ $F\cap D\ne\emptyset$. 
Define the \emph{canonical structure of} $\LL$ to be  the structure
$$
\LL_+=(\fF_\LL,\fI_\LL, \lap,  S_\LL,T_\LL,),
$$
where, for $\vv F\in \fF_\LL{}^n$ and $D\in\fI_\LL$,
$$
\vv F S_\LL D \quad\text{iff \quad there exists $\vv a\in_\pi\vv F$ with } f(\vv a)\in D;
$$
while, for $F\in \fF_\LL$ and $\vv D\in\fI_\LL{}^m$,
$$
 F T_\LL \vv D \quad\text{iff \quad there exists $\vv a\in_\pi\vv D$ with } g(\vv a)\in F.
$$

\begin{lemma}  \label{lesub}
In\/ $\LL_+$, if $F,F'\in\fF_\LL$ then $F\le_1 F'$ iff $F\sub F'$; and if $D,D'\in\fI_\LL$ then $D\le_2 D'$ iff $D\sub D'$.
\end{lemma}
\begin{proof}
If $F\sub F'$, then $F\lap D$ implies $F'\lap D$, so $\rho_\lap\{F\}\sub\rho_\lap\{F'\}$, i.e.\ $F\le_1 F'$ by \eqref{leone}. Conversely, suppose  $\rho_\lap\{F\}\sub\rho_\lap\{F'\}$. If $a\in F$, let $D$ be the ideal $\{b\in\LL:b\leq a\}$ generated by $a$. Then $a\in F\cap D$, so $D\in \rho_\lap\{F\}$, hence there exists $b\in F'\cap D$, so $b\leq a$ and thus $a\in F'$. This shows $F\sub F'$. 
The case of $\le_2$ is the order-dual of this argument.
\end{proof}
Assume from now that $\LL$ is an $\Om$-NLO.

\begin{lemma}
All sections of the relations $ S_\LL$ and $T_\LL$ are stable in $\LL_+$, making $\LL_+$ an $\Om$-polarity.
\end{lemma}

\begin{proof}
First consider a section of the form $S_\LL[\vv F,-]$ with $\vv F\in\fF_\LL{}^n$.
Let $G$ be the filter of $\LL$ generated by $\{f(\vv a):\vv a\in_\pi\vv F\}$.
For any $D\in S_\LL[\vv F,-]$ there exists $\vv a\in_\pi\vv F$ such that $f(\vv a)\in D$. But then $f(\vv a)\in G$, so $G\lap D$. This proves $G\in \lam_\lap S_\LL[\vv F,-]$.
Now take any $D\in\rho_\lap \lam_\lap S_\LL[\vv F,-]$. Then $G\lap D$ so there exists an $d\in G$ with $d\in D$. By definition of $G$ there is a finite subset $Z$ of $\pi\vv F$ such that
\begin{equation}  \label{meetfZld}
\meet \{f(\vv a):\vv a\in  Z\} \leq d.
\end{equation}
For all $i<n$ put $b_i=\meet\{a_i:\vv a\in Z\}\in F_i$, and let $\vv b=(\vc{b}{n})\in_\pi\vv F$.
Then for all $\vv a\in Z$ we have $\vv b\leq\vv a$, so $f(\vv b)\leq f(\vv a)$ as operators are isotone. Hence by \eqref{meetfZld} $f(\vv b)\leq d\in D$, so $f(\vv b)\in D$. As $\vv b\in_\pi\vv F$, this gives $\vv F S_\LL D$, so $D\in S_\LL[\vv F,-]$.
We have now shown that $\rho_\lap \lam_\lap S_\LL[\vv F,-]\sub S_\LL[\vv F,-]$, which is enough to conclude that $S_\LL[\vv F,-]$ is stable.

Next take a section of the form $S_\LL[\vv F[-]_i,D]$ with $D\in\fI_\LL$. Let $E$ be the ideal generated by the set
$$
E_0=\{b\in\LL: \exists \vv a\in_\pi\vv F[\LL/i](f(\vv a[b/i])\in D)\}.
$$
For any $G\in S_\LL[\vv F[-]_i,D]$  we have  $\vv F[G/i]S_\LL D$ so
there exist $\vv a\in \vv F[\LL/i]$ and $b\in G$ such that $\vv a[b/i]\in\vv F[G/i]$ and $f(\vv a[b/i])\in D$.
Then $b\in E_0$, so $G\lap E$. This proves that $E\in\rho_\lap S_\LL[\vv F[-]_i,D]$.

Now let $G\in\lam_\lap \rho_\lap S_\LL[\vv F[-]_i,D]$.
Then there exists  $d\in G\cap E$. By definition of $E$ there is a finite set $Z\sub E_0$ with $d\leq \join Z$. Hence $\join Z\in G$ as $d$ belongs to the filter $G$. For each $b\in Z$ there exists 
$ \vv{a_b}\in_\pi\vv F[\LL/i]$ such that $f(\vv{a_b}[b/i])\in D$.
Now for all $j<n$, put $c_j=\meet\{(a_b)_j : b\in Z\}$. Then $c_j\in F_j$ provided $j\ne i$. Let $\vv c=(\vc{c}{n})$.
Then for all $b\in Z$, $\vv c[b/i]\leq \vv{a_b}[b/i]$, so  $f(\vv c[b/i])\leq f(\vv{a_b}[b/i])\in D$, hence 
$f(\vv c[b/i])\in D$.  \emph{Since $f$ is a normal operator}, we conclude that
$$ \textstyle
f(\vv c[\join Z/i]) =\join\{f(\vv c[b/i]):b\in Z\}\in D.
$$
But $\join Z\in G$, so  $ \vv c[\join Z/i]\in\vv F[G/i]$,  implying that $\vv F[G/i] S_\LL D$ and thus
$G\in S_\LL[\vv F[-]_i,D]$. This proves that 
$\lam_\lap \rho_\lap S_\LL[\vv F[-]_i,D] \sub  S_\LL[\vv F[-]_i,D]$, hence $ S_\LL[\vv F[-]_i,D] $ is stable.

The arguments for the stability sections of $T_\LL$ are essentially the duals of those for $S_\LL$, but we go through the details, first for a section of the form $T_\LL[-,\vv D]$ with $\vv D\in \fI{}^m$.

Let $E$ be the ideal  generated by $\{g(\vv a):\vv a\in_\pi\vv D\}$.
For any $G\in T_\LL[-,\vv D]$ there exists $\vv a\in_\pi\vv D$ such that $g(\vv a)\in G$. But then $g(\vv a)\in E$, so $G\lap E$. This proves $E\in \rho_\lap T_\LL[-,\vv D]$.
Now take any $G\in \lam_\lap \rho_\lap T_\LL[-,\vv D]$. Then $G\lap E$ so there exists an $b\in G$ with $b\in E$. By definition of $E$ there is a finite subset $Z$ of $\pi\vv D$ such that
\begin{equation}  \label{joinZlb}
b\leq \join  \{g(\vv a):\vv a\in  Z\}.
\end{equation}
For all $i<m$ put $d_i=\join\{a_i:\vv a\in Z\}\in D_i$, and let $\vv d=(\vc{d}{n})\in_\pi\vv D$.
Then for all $\vv a\in Z$ we have $\vv a\leq\vv d$, so $g(\vv a)\leq g(\vv d)$ as dual operators are isotone. Hence by \eqref{joinZlb} 
$b\leq g _\LL(\vv d)$. Then $g _\LL(\vv d)\in F$ as $b\in F$.
 As $\vv d\in_\pi\vv D$, this gives $ F T_\LL \vv D$, so $F\in T_\LL[-,\vv D]$.
We have now shown that   $ \lam_\lap \rho_\lap T_\LL[-,\vv D]\sub T_\LL[-,\vv D]$, hence $T_\LL[-,\vv D]$ is stable.

Finally we consider a section of the form   $T_\LL[F,\vv D[-]_i]$ with $F\in\fF_\LL$.  Let $G$ be the filter generated by the set
$$
G_0=\{d\in\LL: \exists \vv a\in_\pi\vv D[\LL/i](g(\vv a[d/i])\in F)\}.
$$
For any $E\in T_\LL[F,\vv D[-]_i]$  we have  $FT_\LL \vv D[E/i]$ so
there exist $\vv a\in \vv D[\LL/i]$ and $d\in E$ such that $\vv a[d/i]\in\vv D[E/i]$ and $g(\vv a[d/i])\in F$.
Then $d\in G_0$, so $G\lap E$.
This proves that $G\in\lam_\lap T_\LL[F,\vv D[-]_i]$.
Now let $E\in \rho_\lap \lam_\lap T_\LL[F,\vv D[-]_i]$.
Then there exists  $b\in G\cap E$. By definition of $G$ there is a finite set $Z\sub G_0$ with $\meet Z\leq b$. Hence $\meet Z\in E$ as $b$ belongs to the ideal $E$.
For each $d\in Z$ there exists 
$ \vv{a_d}\in_\pi\vv D[\LL/i]$ such that $g(\vv{a_d}[d/i])\in F$.
Now for all $j<m$, put $c_j=\join\{(a_d)_j : d\in Z\}$. Then $c_j\in D_j$ provided $j\ne i$. Let $\vv c=(\vc{c}{m})$.
Then for all $d\in Z$, $ \vv{a_d}[d/i]\leq \vv c[d/i]$, so  $g( \vv{a_d}[d/i])\leq g(\vv c[d/i])$,
hence  $g(\vv c[d/i])\in F$.  \emph{Since $g$ is a normal dual operator}, we conclude that
$$ \textstyle
g(\vv c[\meet Z/i]) =\meet\{g(\vv c[d/i]):d\in Z\}\in F.
$$
But $\meet Z\in E$, so  $ \vv c[\meet Z/i]\in\vv D[E/i]$,  implying that $F T_\LL \vv D[E/i]$ and thus
$E\in T_\LL[F,\vv D[-]_i]$. This proves that 
$ \rho_\lap \lam_\lap T_\LL[F,\vv D[-]_i] \sub  T_\LL[F,\vv D[-]_i]$, hence $ T_\LL[F,\vv D[-]_i] $ is stable.
\end{proof}

\begin{theorem}\label{thcano}
 $(\LL_+)^+$ is a canonical extension of\/ $\LL$ as a lattice.
\end{theorem}
\begin{proof}
This was shown by Gehrke and Harding \cite{gehr:boun01}, building on work on lattice representations by  Urquhart \cite{urqu:topo78} and Hartung \cite{hart:topo92}. We give details of a proof here, as we make  further use of its ideas.

For $a\in\LL$, define $\fF_a=\{F\in\fF_\LL:a\in F\}$ and $\fI_a=\{I\in \fI_\LL:a\in I\}$. 
Then $\fF_a=\lam_{\lap} \fI_a$ and $\fI_a=\rho_{\lap} \fF_a$. For the first equation, if $F\in\fF_a$, then any $D\in \fI_a$ has $a\in F\cap D$ so $F\lap D$, hence $F\in \lam_{\lap} \fI_a$. For the converse, let $(a]=\{b\in\LL:b\leq a\}\in\fI_a$ be the ideal generated by $a$. Then any $F\in \lam_{\lap} \fI_a$ has $F\lap(a]$, hence $a\in F$ and so $F\in\fF_a$. The second equation is similar.

Thus $\fF_a=\lam_{\lap} \fI_a=\lam_{\lap} \rho_{\lap} \fF_a$, showing $\fF_a$ is stable.
The map 
$\thet(a)=\fF_a$ gives a lattice embedding of $\LL$ into the stable set lattice of the polarity $(\fF_\LL,\fI_\LL,\lap)$, hence into $ (\LL_+)^+$. To show this, observe that $\fF_{a\land b}=\fF_a\cap\fF_b$, so $\thet$ preserves binary meets. Since $\fF_1=\fF_\LL$ and $\fF_0=\lam_{\lap} \fI_0=\lam_{\lap} \fI_\LL$, it preserves the universal bounds.
Also 
$\fI_{a\lor b}=\fI_a\cap\fI_b$,  so $\fF_{a\lor b}=\lam_\lap(\fI_a\cap\fI_b)=\lam_\lap(\rho_\lap\fF_a\cap\rho_\lap\fF_b)
=\lam_\lap\rho_\lap(\fF_a\cup\fF_b)=\fF_a\lor\fF_b$, so $\thet$ preserves binary joins. Moreover, if $a\nleqslant b$, then the filter $[a)=\{b'\in\LL:a\leq b'\}$ belongs to $\thet(a)\setminus\thet(b)$, so $\thet$ is an order-embedding.

Next we show that $\thet\colon\LL\to (\LL_+)^+$ is a compact and dense embedding. For compactness it suffices to take any subsets $Z,W$ of $\LL$ such that $\bigcap \thet[Z]\sub\join \thet[W]$ in $(\LL_+)^+$, and show that there are finite sets $Z'\sub Z$ and $W'\sub W$ with $\meet Z'\leq\join W'$ \cite[2.4]{gehr:boun01}. Given such $Z$ and $W$, let $F$ be the filter of $\LL$ generated by $Z$. Then $F\in \bigcap \thet[Z]\sub\join \thet[W]$, so
$$\textstyle
F\in \join \thet[W] =\lam_\lap\rho_\lap\bigcup\thet[W] = \lam_\lap\bigcap_{b\in W}\rho_\lap \fF_b
= \lam_\lap\bigcap_{b\in W}\fI_b.
$$
Now if $D$ is the ideal generated by $W$, then $D\in\bigcap_{b\in W}\fI_b$, so then $F\lap D$. This means there
is some $a\in F\cap D$ and so by the nature of generated filters and ideals there are finite sets $Z'\sub Z$ and $W'\sub W$ with $\meet Z'\leq a\leq \join W'$, hence $\meet Z'\leq\join W'$ as required.

Density of $\thet$  requires that each member of $(\LL_+)^+$ is both a join of meets and a meet of joins of members of 
$\thet[\LL]$. We use the fact \eqref{joinmeetdense} that in any polarity, a member of $P^+$ is both a join of elements of the form $\lam\rho\{x\}$ and a meet of elements of the form $\lam\{y\}$.

If $F\in \fF_\LL$ and $D\in\fI_\LL$, then  $F\lap D$ 
iff   $\exists a\in F(D\in \fI_a)$. So
$\rho_\lap\{F\}=\bigcup_{a\in F}\fI_a$.
Hence
$\lam_\lap\rho_\lap\{F\}=\lam_\lap\bigcup_{a\in F}\fI_a=\bigcap_{a\in F}\lam_\lap\fI_a=\bigcap_{a\in F}\fF_a
=\bigcap_{a\in F}\thet(a)$. Combining this with \eqref{joinmeetdense} gives that if $A\in (\LL_+)^+$, then
\begin{equation}\label{joinmeet}
A=\join\nolimits_{F\in A}\lam_\lap\rho_\lap\{F\}=\join\nolimits_{F\in A}\bigcap\nolimits_{a\in F}\thet(a).
\end{equation}
Also, $F\lap D$ 
iff   $\exists a\in D(F\in \fF_a)$, so $\lam_\lap\{D\}=\bigcup_{a\in D}\thet(a)$. Since $\lam_\lap\{D\}$ is stable, this union is a join. Together with  \eqref{joinmeetdense} we then get that if $A\in (\LL_+)^+$, then
\begin{equation}\label{meetjoin}
A=\bigcap\nolimits_{D\in\rho_\lap A}\join\nolimits_{a\in D}\thet(a).
\end{equation}
\eqref{joinmeet} and \eqref{meetjoin} show that $\thet$ is dense as required.
\end{proof}

\begin{theorem}  \label{Lplusplus}
 $(\LL_+)^+$ is a canonical extension of\/ $\LL$ as an $\Om$-lattice.
\end{theorem}

\begin{proof}
We  need to supplement Theorem \ref{thcano} by showing that its embedding $\thet$ preserves $f$ and $g$, and that the operations $f_{S_\LL}$ and $g_{T_\LL}$ on $(\LL_+)^+$ are the canonical extensions of $f$ and $g$ as defined in \eqref{floclosed}--\eqref{fup}. 
We will denote $(\LL_+)^+$ more briefly as $\LL^\sg$, as justified by Theorem \ref{thcano}.

Preservation of $f$ requires that for any $\vv a\in\LL^n$,
\begin{equation}\label{thetpresf}
\thet(f(\vv a) )= f_{S_\LL}(\thet(\vv a)),
\end{equation}
i.e.\ $\fF_{f(\vv a)}=\lam_\lap  f^\bullet_{S_\LL}(\thet(\vv a))$, where $f^\bullet_{S_\LL}(\thet(\vv a))=
\{D\in\fI_\LL:(\pi\thet(\vv a))S_\LL D\}$. 
It is enough to show that
\begin{equation}  \label{fSLI}
f^\bullet_{S_\LL}(\thet(\vv a)) = \fI_{f(\vv a)},
\end{equation}
since that implies that  $\lam_\lap f^\bullet_{S_\LL}(\thet(\vv a)) = \lam_\lap\fI_{f(\vv a)}=\fF_{f(\vv a)}$, as desired.
Note that
$$
\pi\thet(\vv a)=\fF_{a_0}\times\cdots\times \fF_{a_{n-1}}=\{\vv F\in\fF_\LL{}^n:\vv a\in_\pi\vv F\},
$$
so \eqref{fSLI} amounts to the claim that for any $D\in\fI_\LL$,
\begin{equation}\label{truthforf}
f(\vv a)\in D \quad\text{iff}\quad \forall \vv F\in\fF_\LL{}^n \big(\vv a\in_\pi\vv F \text{ implies } \vv FS_\LL\vv D\big).
\end{equation}
If $f(\vv a)\in D $, then if  $\vv a\in_\pi\vv F$ it is immediate that  $\vv FS_\LL D$ by definition of $S_\LL$.
Conversely, for each $i<n$, let $F_i=[a_i)$, the filter of $\LL$ generated by $a_i$, so that $F_i\in\fF_{a_i}$, and let  $\vv F=(\vc{F}{n})$. Then $\vv a\in_\pi\vv F$, so if the right side of \eqref{truthforf} holds then $\vv FS_\LL D$, hence there exists some $\vv b\in_\pi\vv F$ such that $f(\vv b)\in D$.
Then $\vv a\leq\vv b$. But any operator is isotone, so this implies $f(\vv a)\leq f(\vv b)\in D$, hence $f(\vv a)\in D$. That completes the proof of \eqref{truthforf}, and therefore of \eqref{thetpresf}.

Next we dualise this argument to show that for any $\vv a\in\LL^m$,
\begin{equation}\label{thetpresg}
\thet(g(\vv a) )= g_{T_\LL}(\thet(\vv a)),
\end{equation}
i.e.\ $\fF_{g(\vv a)}=\{F\in\fF_\LL:  F T_\LL\pi\rho_\lap\thet(\vv a)\}$. Note that
$$
\pi\rho_\lap\thet(\vv a)=\rho_\lap\fF_{a_0}\times\cdots\times\rho_\lap\fF_{a_{m-1}}=\{\vv D\in\fI_\LL{}^m:\vv a\in_\pi\vv D\},
$$
so what we want for \eqref{thetpresg} is that for any $F\in\fF_\LL$, 
\begin{equation}\label{truthforg}
g(\vv a)\in F \quad\text{iff}\quad \forall \vv D\in\fI_\LL{}^m \big(\vv a\in_\pi\vv D \text{ implies } FT_\LL\vv D\big).
\end{equation}
Now if $g(\vv a)\in F$, then if $\vv a\in_\pi\vv D$ it is immediate that $ FT_\LL\vv D$ by definition of $T_\LL$.
For the converse, for each $i<m$ let $D_i=(a_i]$, the ideal of $\LL$ generated by $a_i$, and put $\vv D=(\vc{D}{m})$.
Then $\vv a\in_\pi\vv D$, so if the right side of \eqref{truthforg} holds then $fT_\LL\vv D$, hence there exists $\vv b\in_\pi\vv D$ with $g(\vv b)\in F$. Then $g(\vv b)\leq g(\vv a)$ as dual operators are isotone, hence $g(\vv a)\in F$ as $F$ is a filter. This proves \eqref{truthforg} and hence \eqref{thetpresg}.

Now we want to show that the lower canonical extension of $f$ on $\LL^\sg=(\LL_+)^+$ is just  $f_{S_\LL}$, i.e.\
$f\lo(\vv Z)=f_{S_\LL}(\vv Z)$ for all $\vv Z\in (\LL^\sg)^n$. First we show this for the case that $\vv Z$ is any \emph{closed} element of $(\LL^\sg)^n$, which means that for all $i<n$, $Z_i$ is a closed element of $\LL^\sg$, hence there is some subset $A_i\sub\LL$ such that
\begin{equation}\textstyle  \label{Ziclosed}
Z_i=\bigcap\thet[A_i]=\bigcap\{\fF_a:a\in A_i\}=\{F\in\fF_\LL:A_i\sub F\}.
\end{equation}
Since $\vv Z\in K((\LL^\sg)^n)$,  \eqref{floclosed} gives
\begin{align}
f\lo(\vv Z)&=\bigcap\{\thet(f(\vv a)):\vv a\in\LL^n \text{ and }\vv Z\sub_\pi\thet(\vv a)\}.  \label{floZbigcap}
\end{align}
Now $f_{S_\LL}$ is isotone, being an operator, so if $\vv Z\sub_\pi\thet(\vv a)$, then
$$
f_{S_\LL}(\vv Z)\sub f_{S_\LL}\thet(\vv a)=\thet(f(\vv a))
$$
by \eqref{thetpresf}. Hence $f_{S_\LL}(\vv Z)\sub \bigcap\{\thet(f(\vv a)):\vv Z\sub_\pi\thet(\vv a)\}=f\lo(\vv Z)$
by \eqref{floZbigcap}.

For the converse inclusion, suppose $G\in  f\lo(\vv Z)$.
For all $i<n$, let $F_i$ be the filter generated by $A_i$. Then $F_i\in Z_i$ by \eqref{Ziclosed}, so the tuple $\vv F=(\vc{F}{n})$ belongs to $\pi\vv Z$. Thus for any $D\in f^\bullet_{S_\LL}(\vv Z)$ we have $\vv FS_{\LL}D$ and so there is some $\vv a\in_\pi\vv F$ such that $f(\vv a)\in D$.
But $Z_i=\{F\in\fF_\LL:F_i\sub F\}$,  by \eqref{Ziclosed} and the definition of $F_i$, so as $a_i\in F_i$, we get $a_i\in\bigcap Z_i$, so $Z_i\sub\fF_{a_i}=\thet(a_i)$. 
Thus $\vv Z\sub_\pi\thet(\vv a)$, so by \eqref{floZbigcap} $f\lo(\vv Z)\sub \thet(f(\vv a))$.
As $G\in  f\lo(\vv Z)$, this gives $f(\vv a)\in G$. But  $f(\vv a)\in D$, so $G\lap D$. Altogether this proves that
$G\in\lam_\lap f^\bullet_{S_\LL}(\vv Z)=f_{S_\LL}(\vv Z)$.

That completes the proof that  $f\lo$ and $f_{S_\LL}$ agree on all closed members of $ (\LL^\sg)^n$. To show that they agree on an arbitrary $\vv Z\in  (\LL^\sg)^n$ we use the fact that each $Z_i$ is a join of closed members of $\LL^\sg$, so $Z_i=\join\Z_i$ for some $\Z_i\sub K(\LL^\sg)$. Now $f\lo$ is a complete normal operator as it is the lower extension of a normal operator $f$ \cite[Sec.~4]{gehr:boun01}, and $f_{S_\LL}$ is a complete normal operator by Theorem \ref{fScomplop}, so we reason that
\begin{align*}
f\lo(\vv Z)
&= \textstyle
f\lo(\vc{\join\Z}{n}) \textstyle
\\ 
&=\textstyle
\join\{ f\lo(\vv Z'): Z'_i\in\Z_i \text{ for all }i<n\} \qquad \text{by \eqref{joincomplete} for $f\lo$,}
\\
&=\textstyle
\join\{ f_{S_\LL}(\vv Z'): Z'_i\in\Z_i \text{ for all }i<n\}  \qquad \text{as $f\lo= f_{S_\LL}$ on } K((\LL^\sg)^n),
\\
&=\textstyle
f_{S_\LL}(\vc{\join\Z}{n}) \qquad \text{by \eqref{joincomplete} for $f_{S_\LL}$},
\\
&= f_{S_\LL}(\vv Z).
\end{align*}

Finally, we show that $g_{T_\LL}$ is the upper canonical extension $g\up$. First we prove that
$g\up(\vv Z)=g_{T_\LL}(\vv Z)$ whenever $\vv Z$ is any \emph{open} element of $(\LL^\sg)^n$, which means that for all $i<m$, $Z_i$ is an open element of $\LL^\sg$, so $Z_i=\join\thet[B_i] $ for some $B_i\sub\LL$. Hence by the definition of $\join$ in $\LL^\sg$,
\begin{equation}\textstyle   \label{Ziopen}
Z_i=\lam_\lap\rho_\lap\bigcup_{a\in B_i}\fF_{a}=
\lam_\lap\bigcap_{a\in B_i}\fI_{a}=
\lam_\lap\{D\in\fI_\LL:B_i\sub D\}.
\end{equation}
Since $\vv Z\in O((\LL^\sg)^n)$,  \eqref{fupopen} gives
\begin{equation}\label{gupZjoin}
g\up(\vv Z)=\join\{\thet(g(\vv a)):\vv a\in\LL^n \text{ and }\thet(\vv a)\sub_\pi\vv Z\}.
\end{equation}
But if $\thet(\vv a)\sub_\pi\vv Z$,  then $\thet(g(\vv a))=g_{T_\LL}(\thet(\vv a))\sub g_{T_\LL}(\vv Z)$ as the dual operator $g_{T_\LL}$ is isotone. Hence we get $g\up(\vv Z)\sub g_{T_\LL}(\vv Z)$ by \eqref{gupZjoin}.

For the converse inclusion, suppose $F\in g_{T_\LL}(\vv Z)$. For $i<m$, let $D_i$ be the ideal of $\LL$ generated by $B_i$.
Then by \eqref{Ziopen},  $\rho_\lap Z_i=\rho_\lap\lam_\lap\{D\in\fI_\LL:B_i\sub D\}$, and so as $B_i\sub D_i$ we get $D_i\in\rho_\lap Z_i$.
Thus if $\vv D=(\vc{D}{m})$, then $\vv D\in\pi\rho_\lap\vv Z$, so $FT_\LL\vv D$ as $F\in g_{T_\LL}(\vv Z)$.
Hence there is some $\vv a\in\vv D$ such that $g(\vv a)\in F$.
For each $i<m$, we have $a_i\in D_i$ and so any filter containing $a_i$ intersects every ideal including $D_i$, i.e. $\thet(a_i)\sub \lam_\lap\{D\in\fI_\LL:D_i\sub D\}=Z_i$. Thus $\thet(\vv a)\sub_\pi\vv Z$, implying 
$\thet(g(\vv a))\sub g\up(\vv Z)$ by \eqref{gupZjoin}.  But $F\in\thet(g(\vv a))$, so then $F\in g\up(\vv Z)$.

That completes the proof that  $g\up$ and $g_{T_\LL}$ agree on all open members of $ (\LL^\sg)^n$. Since every member of  $ (\LL^\sg)^n$ is a meet of open members, and $g\up$ and $g_{T_\LL}$ are both complete normal dual operators preserving all meets in each coordinate, we can then show that  $g\up$ and $g_{T_\LL}$ are identical by using the order dual of \eqref{joincomplete}.
\end{proof}

Theorem \ref{Lplusplus} justifies the equation $\LL^\sg=(\LL_+)^+$. In the case that $\LL$ is the stable set lattice $P^+$ of an $\Om$-polarity, we will call the canonical structure $(P^+)_+$ the \emph{canonical extension of} $P$. Its stable lattice
$((P^+)_+)^+$ is the canonical extension $(P^+)^\sg$ of the $\Om$-lattice $P^+$.

\section{Dual Categories}

At the end of Section \ref{secbmorph} it was shown that there is a contravariant functor from $\opol$ to $\onlo$. We now construct such a functor in the reverse direction.

Let $\thet\colon (\LL,f_\LL,g_\LL)\to(\M,f_\M,g_\M)$   be an $\Om$-homomorphism between two $\Om$-NLO's.  If $E$ is a filter or ideal of $\M$, then $\thet\inv E$ is a filter or ideal of $\LL$, respectively, so we can define
 $\al_\thet:\fF_\M\to\fF_\LL$ and 
$\be_\thet:\fI_\M\to\fI_\LL$ by putting 
$\al_\thet F=\thet\inv F$ and $\be_\thet D=\thet\inv D$.

\begin{theorem}  \label{morphdual}
The pair $\thet_+=(\al_\thet,\be_\thet)$ is a bounded morphism from $\M_+$ to $\LL_+$.
  If $\thet$ is injective,  $\al_\thet$ and $\be_\thet$ are  surjective.
   If $\thet$ is surjective,  then $\al_\thet$ and $\be_\thet$ are injective and $\thet_+$ is an isomorphism from $\M_+$ to the inner substructure $\im\thet_+$ of\/ $\LL_+$.

\end{theorem}
\begin{proof}
Since $\thet\inv$ preserves set inclusion,  $\al_\thet$ and $\be_\thet$ are isotone by Lemma \ref{lesub}. We show that they fulfils the conditions of Definition \ref{defmorph}, with $R=\lap$.

(1$_R$): Let $F\in\fF_\M$ and $D\in\fI_\M$ have $\al_\thet F\lap \be_\thet D$. Then there is some $a\in\thet\inv F\cap\thet\inv D$. Then $\thet a\in F\cap D$, showing that $F\lap D$.

(2$_R$): Let $F\in\fF_\LL$ and $D\in\fI_\M$ have $\al_\thet\inv[F)_1\lap D$, where $[F)_1=\{G\in\fF_\LL:F\sub G\}$. We have to show that $F\lap \be_\thet (D)$.

Now as $\thet$ preserves finite meets, the subset $\thet[F]$ of $\M$ is closed under finite meets. Hence the filter it generates is its upward closure in $\M$, i.e.\ the set
\begin{equation}  \label{filgenF}
G=\{b\in\M: \exists a\in F(\thet(a)\leq b)\}.
\end{equation}
Since $\thet[F]\sub G$ we have $F\sub\thet\inv(G)=\al_\thet(G)$, and so $G\in \al_\thet\inv[F)_1$. But
 $\al_\thet\inv[F)_1\lap D$, so there exists $b\in G\cap D$, and so for some $a\in F$, 
$ \thet(a)\leq b\in D$. As $D$ is an ideal, this gives  $\thet(a)\in D$. Thus $a\in F\cap\thet\inv(D)=F\cap \be_\thet (D)$, giving the desired result $F\lap \be_\thet (D)$.

(3$_R$): This is just the order-dual of the argument for (2$_R$).
 Let $F\in\fF_\M$ and $D\in\fI_\LL$ have $F\lap \be_\thet\inv[D)_2$, where $[D)_2=\{E\in\fI_\LL:D\sub E\}$. We have to show that $\al_\thet(F)\lap D$.
 
The subset $\thet[D]$ of $\M$ is closed under finite joins, so the ideal it generates is 
$$
E=\{b\in\M: \exists a\in D(b\leq\thet(a))\}.
$$
Since $\thet[D]\sub E$ we have $D\sub\thet\inv(E)=\be_\thet(E)$, and so $E\in \be_\thet\inv[D)_2$. But
$F\lap \be_\thet\inv[D)_2$, so there exists $b\in F\cap E$, and so for some $a\in D$, 
$b\leq\thet(a)$.  As $F$ is a filter containing $b$, this implies $\thet(a)\in F$, hence
$a\in \thet\inv(F)\cap D= \al_\thet(F)\cap D$, showing $\al_\thet(F)\lap D$ as desired.

(1$_S$):  Let $\vv F\in\fF_\M{}^n$ and $D\in\fI_\M$ have $\al_\thet (\vv F)S_{f_\LL} \be_\thet D$. Then there is some $\vv a\in_\pi\al_\thet (\vv F)=\thet\inv (\vv F)$ with  $f_\LL(\vv a)\in \be_\thet D$.
Then $\thet(\vv a)\in_\pi\vv F$, and  $f_\M(\thet(\vv a))= \thet(f_\LL(\vv a))\in D$, showing that $\vv F S_{f_\M} D$.

(2$_S$):  Let  $\al_\thet\inv[\vv{F})_1S_{f_\M} D$, where $\vv{F}\in\fF_\LL{}^n$ and $D\in\fI_\M$. For all $i<n$, let $G_i$ be the filter of $\M$ generated by $\thet[F_i]$. As in the case of (2$_R$), we have
\begin{equation} \label{defGi}
G_i=\{b\in \M: \exists a\in F_i\,(\thet(a)\leq b)\}.
\end{equation}
Let $\vv G=(\vc{G}{n})$. Since in general $F_i\sub\thet\inv G_i=\al_\thet(G_i)\in\fF_\LL$, we get 
$\vv{F}\sub_\pi\al_\thet(\vv G)$, hence $\vv G\in \al_\thet\inv[\vv{F})_1$. But  $\al_\thet\inv[\vv{F})_1S_{f_\M} D$, 
so there must exist $\vv b\in_\pi\vv G$ such that $f_\M(\vv b)\in D$.
 For all $i<n$, as $b_i\in G_i$, by \eqref{defGi} there exists $a_i\in F_i$ with $\thet(a_i)\leq b_i$. Putting $\vv a=(\vc{a}{n})$, we have
 $\thet(\vv a)\leq\vv b$, hence   $f_\M(\thet(\vv a))\leq f_\M(\vv b)\in D$, and so $f_\M(\thet(\vv a))\in D$, i.e.\ 
 $\thet(f_\LL(\vv a))\in D$. Thus $f_\LL(\vv a)\in\thet\inv(D)=\be_\thet(D)$. But $\vv a\in_\pi\vv{F}$, so this proves
 $\vv{F}S_{f_\LL} \be_\thet(D)$ as required.

(1$_T$):  Let $\al_\thet  (F)T_{g_\LL} \be_\thet(\vv D)$, where  $ F\in\fF_\M$ and $\vv D\in\fI_\M{}^m$, so that there is some $\vv a\in_\pi \be_\thet(\vv D)$ with $g_\LL(\vv a)\in\al_\thet (F)$. 
Then $\thet(\vv a)\in_\pi\vv D$, and  $g_\M(\thet(\vv a))= \thet(g_\LL(\vv a))\in F$, showing that $ F T_{g_\M} \vv D$.

(2$_T$):  Let   $F T_{g_\M} \be_\thet\inv[\vv{D})_2$, where $F\in\fF_\M$ and $\vv D\in\fI_\LL{}^m$. Let $E_i$ be the ideal of $\M$ generated by $\thet[D_i]$, so that
$$
E_i=\{b\in\M: \exists a\in D_i(b\leq\thet(a))\}.
$$
Let $\vv E=(\vc{E}{m})$.  Then $D_i\sub\thet\inv(E_i)=\be_\thet(E_i)$ for all $i<m$, so $\vv D\sub_\pi\be_\thet(\vv E)$, hence $\vv E\in\be_\thet\inv[\vv D)_2$, and therefore $FT_{g_\M}(\vv E)$.
So there must exist $\vv b\in_\pi\vv E$ such that $g_\M(\vv b)\in F$.
 For all $i<m$, as $b_i\in E_i$  there exists $a_i\in D_i$ with $b_i\leq \thet(a_i)$. Putting $\vv a=(\vc{a}{m})$, we have
 $\vv b\leq\thet(\vv a)$, hence   $g_\M(\vv b)\leq g_\M(\thet(\vv a))$.
  As $F$ is a filter containing $g_\M(\vv b)$, this implies $ g_\M(\thet(\vv a))\in F$,  i.e.\ $\thet(g_\LL(\vv a))\in F$.
 Thus $g_\LL(\vv a))\in\al_\thet(F)$. Since $\vv a\in_\pi\vv D$, this proves  $\al_\thet(F)T_{g_\LL}\vv D$ as required.
 
 That completes the proof that the pair $\al_\thet,\be_\thet$ is a bounded morphism. 
  Now suppose that  $\thet$ is injective. To show $\al_\thet$ is surjective, take any $F\in\fF_\LL$. Let $G$ be the filter of $\M$ generated by $\thet[F]$, as given in \eqref{filgenF}. Then $F\sub\thet\inv G$. To prove the converse inclusion, let $b\in\thet\inv G$. Then by \eqref{filgenF},  there exists $a\in F$  such that $\thet(a)\leq\thet(b)$. But $\thet$, being an injective homomorphism, is order invariant, so this implies $a\leq b$, hence $b\in G$.
 Thus  $G=\thet\inv F=\al_\thet F$, showing that   $\al_\thet:\fF_\M\to\fF_\LL$ is surjective.
 A dual argument shows that $\be_\thet:\fI_\M\to\fI_\LL$ is surjective: if $E\in\fI_\LL$, then $E=\be_\thet D$ where $D$ is the ideal of $\M$ generated by $\thet[E]$.
 
 Finally, suppose that $\thet$ is surjective. If $\al_\thet(F)=\al_\thet(G)$, then
 $F= \thet[\thet\inv F]=\thet[\thet\inv G]=G$, so $\al_\thet$ is injective. Similarly  $\be_\thet$ is injective.
Then by Theorem \ref{converseback}, to show that $\thet_+$ makes $\M_+$ isomorphic to $\im\thet_+$, it suffices to show that it preserves the relations.

Preservation of $R$: Let $F\in\fF_\M$ and $D\in\fI_\M$ have  $ F\lap  D$. Then there is some $b\in F\cap D$. But $b=\thet(a)$ for some $a$ as $\thet$ is surjective. Then $a\in\thet\inv F\cap\thet\inv G$, so $\al_\thet F\lap \be_\thet D$.

Preservation of $S$:  Let $\vv F\in\fF_\M{}^n$ and $D\in\fI_\M$ have $\vv F S_{f_\M} D$. Then there is some 
$\vv b\in_\pi\vv F$ with $f_\M(\vv b)\in D$. But  $\vv b=\thet(\vv a)$ for some $\vv a$. Then $\vv a\in\thet\inv(\vv F)=\al_\thet (\vv F)$, and
$\thet(f_\LL(\vv a))=f_\M(\vv b)$, so $f_\LL(\vv a)\in \thet\inv(D)= \be_\thet D$.
Hence $\al_\thet (\vv F)S_{f_\LL} \be_\thet D$. 
 
 The preservation of $T$ is similar.
\end{proof}

Using this result we can infer that the mappings $\LL\mapsto\LL_+$ and $\thet\mapsto\thet_+$ give a contravariant functor from $\onlo$ to $\opol$.

\section{Direct Sums}

Let   $\{P_j: j\in J\}$ be an indexed set of $\Om$-polarities, with $P_j=( X_j, Y_j, R_j,S_j,T_j)$. We define a structure
$$\textstyle
\sum_JP_j=(\sum_J X_j, \sum_J Y_j, R_J,S_J,T_J)
$$
whose stable set lattice $(\sum_JP_j)^+$ is isomorphic to the direct product $\prod_J P_j^+$ of the stable set lattices of the $P_j$'s. This $\sum_JP_j$ is called the \emph{direct sum} of the $P_j$'s.  The polarity part of its definition is due to Wille \cite{will:rest82,will:subd87} and is given also  in \cite[p.184]{gant:form99}.

For each $j\in J$, let $\dot X_j=X_j\times\{j\}$ and $\dot Y_j=Y_j\times\{j\}$. Then  
  $\sum_J X_j=\bigcup_J\dot X_j$   is
 the disjoint union of the $X_j$'s, and   $\sum_J Y_j=\bigcup_J\dot Y_j$ is the disjoint union of the $Y_i$'s.
However, unlike the case of Kripke modal frames, $R$ is not the disjoint union of the $R_j$'s. Rather, we put
\begin{align*}
\dot R_j&=\{((x,j),(y,j)): xR_jy\},
\\
\dot S_j&=  \{((x_0,j),\dots,(x_{n-1},j),(y,j)):   \vv xS_jy\},
\\
\dot T_j&=  \{((x,j),(y_0,j),\dots,(y_{m-1},j)):  xT_j\vv y\},
\end{align*}
and then define
\begin{align*} 
R_J&=  \bigcup\nolimits_J  \dot R_j\ \cup\ \bigcup\{\dot X_j\times \dot Y_k : j\ne k\},
\\
S_J&=  \bigcup\nolimits_J  \dot S_j\ \cup\ \bigcup\{\dot X_{j_0}\times\cdots\times \dot X_{j_{n-1}} \times \dot Y_k : 
(\exists i<n)\, j_i\ne k\},
\\
T_J&=  \bigcup\nolimits_J  \dot T_j\ \cup\ \bigcup\{ \dot X_k\times \dot Y_{j_0}\times\cdots\times \dot Y_{j_{m-1}}  : 
(\exists i<m)\, j_i\ne k\}.
\end{align*}
Spelling this out, we have that $(x,j)R_J(y,k)$ iff either $j\ne k$ or else $j=k$ and $xR_k y$. Likewise,
$((x_0,j_0),\dots,(x_{n-1},j_{n-1}))S_J(y,k)$ iff either $ j_i\ne k$ for some  $i<n$, or else  
$(\vc{x}{n})S_k y$. The description of $T_J$ is similar.

\begin{lemma}
$\sum_JP_j$ is an $\Om$-polarity.
\end{lemma}
\begin{proof}
This requires that all sections of $S_J$ and  $T_J$ are stable. We prove stability for any section of the form
$S_J[\vv{x_J}[-]_i,(y,k)]$, where 
$$\textstyle
\vv{x_J}=((x_0,j_0),\dots,(x_{n-1},j_{n-1}))\in (\sum\nolimits_JX_j)^n,
$$
 $(y,k)\in\sum_JY_j$, and $i<n$.  Let $\vv x=(\vc{x}{n})$.
 
If an element $(u,l)$ of $\sum_JX_j$ is not in $S_J[\vv{x_J}[-]_i,(y,k)]$, then not $\vv{x_J}[(u,l)/i]S_J(y,k)$, so by definition of $S_J$
we have that $\{\vc{j}{i},l,j_{i+1},\dots,j_{n-1}\}=\{k\}$ and not $\vv x[u/i]S_k y$.
So $u$ is not in the section $S_k[\vv x[-]_i,y]$, which is stable in $P_j$.
Hence there is some $z\in Y_k$ with $z\in\rho_{R_k} S_k[\vv x[-]_i,y]$ but not $uR_k z$. Now we show that
\begin{equation}  \label{zkrhoS}
(z,k)\in\rho_{R_J} S_J[\vv{x_J}[-]_i,(y,k)].  
\end{equation}
Take any $(w,j)\in \sum_JX_j$ with $(w,j)\in S_J[\vv{x_J}[-]_i,(y,k)]$. If $j\ne k$, then $(w,j)R_J(z,k)$ by definition of $R_J$.
If $j=k$, then $w\in X_k$ and as $\vv{x_J}[(w,j)/i]S(y,k)$ we have $\vv x[w/i]S_ky$. Since
$z\in\rho_{R_k} S_k[\vv x[-]_i,y]$, this implies $wR_k z$, and so $(w,j)R_J(z,k)$ again. That proves \eqref{zkrhoS}.
But not $uR_kz$ and $l=k$, so not  $(u,l)R_J(z,k)$. By \eqref{zkrhoS} then,   $(u,l)\notin\lam_{R_J}\rho_{R_J}S_J[\vv{x_J}[-]_i,(y,k)]$, completing the proof that $S_J[\vv{x_J}[-]_i,(y,k)]$ is stable.

The cases of other sections of $S_J$, and those of $T_J$, are similar to this.
\end{proof}

Now for each $k\in J$, define functions $\al_k\colon X_k\to\sum_JX_j$ and  $\be_k\colon Y_k\to\sum_JY_j$ by putting
$\al_k(x)=(x,k)$ and $\be_k(y)=(y,k)$.

\begin{lemma}
The pair $\al_k,\be_k$ is a bounded morphism $P_k\to\sum_JP_j$  whose image is an inner substructure of $\sum_JP_j$ isomorphic to $P_k$.
\end{lemma}

\begin{proof}
First we show that $\al_k$ is isotone. Let $\le_1^k$ be the quasi-order of $X_k$ determined by $R_k$, and $x\le_1^k x'$, i.e.\ $\rho_{R_k}\{x\}\sub \rho_{R_k}\{x'\}$. We have to show that  $\rho_{R_J}\{\al_k(x)\}\sub \rho_{R_J}\{\al_k(x')\}$. But if 
$(x,k)R_J(y,j)$ then either $k\ne j$ and hence $(x',k)R_J(y,j)$, or else $k=j$ and $xR_k y$, hence $x'R_k y$ as $x\le_1^k x'$, which again gives  $(x',k)R_J(y,j)$. The proof that $\be_k$ is isotone is similar.

Now given any $\vv x\in X_k^n$ and $y\in Y_k$, the definitions of $\al_k$, $\be_k$ and $S_J$ make it immediate that
$\al_k(\vv x)S_J\be(y)$ iff $\vv xS_k y$. So $\al_k,\be_k$ satisfies (1$_S$) and preserves $S$.

To show that (2$_S$) is satisfied, suppose that not $\vv{x_J} S_J \be_k(y)$, where $y\in Y_k$ and $\vv{x_J}=((x_0,j_0),\dots,(x_{n-1},j_{n-1}))$. Then $j_i=k$ for all $i<n$, and not $\vv x S_k y$, where $\vv x=(\vc{x}{n})$.
Then $\al_k(\vv x)=\vv{x_J}$, so $\vv x\in \al_k\inv[\vv{x_J})_1$, hence not  $\al_k\inv[\vv{x_J})_1S_k y$ as required by (2$_S$).

The proofs that $\al_k,\be_k$ satisfies the other back and forth conditions and also preserves $R$ and $T$ are similar to the above. Thus $\al_k,\be_k$ is a bounded morphism. Since $\al_k$ and $\be_k$ are both injective, the rest of this theorem follows by Theorem \ref{converseback}.
\end{proof}

\begin{theorem} \label{sumprod}
$(\sum_JP_j)^+$ is isomorphic to  $ \prod_J P_j^+.$
\end{theorem}
\begin{proof}
For each $k\in J$,  the bounded morphism $\al_k,\be_k$ induces a  $\Om$-homomorphism
$\thet_k:(\sum_JP_j )^+ \to P_k^+$ by Theorem \ref{homdual}.
The direct product of these $\thet_k$'s is the  $\Om$-homomorpism
$\thet:(\sum_JP_j )^+ \longrightarrow \prod_J P_j^+$ defined by $\thet(A)(k)=\thet_k A=\al_k\inv A$.
$\thet$ is injective, for suppose $\thet(A)=\thet(B)$ and take any $(x,k)\in \sum_J X_j$. Then
$\thet(A)(k)=\thet(B)(k)$, i.e.\ $\al_k\inv A=\al_k\inv B $, so $(x,k)\in A$ iff $(x,k)\in B$. Hence $A=B$. 

Thus if $\thet$ is also surjective, it provides the desired isomorphism. To prove this surjectivity, let 
$\langle B_j:j\in J\rangle$ be any member of $\prod_I P_i^+$. Put
$B=\bigcup_J\al_j[B_j]\sub\sum_JX_j$. If $B\in (\sum_IP_i )^+$, then $\thet(B)$ is defined with
$\thet(B)(j)=\al_j\inv B=B_j$ for all $j$, hence $\thet (B)=\langle B_j:j\in J\rangle$.

Thus it remains to prove that $B$ is stable. Take any $(x,k)\in \sum_JX_j$ with $(x,k)\notin B$. We want 
$(x,k)\notin \lam_{R_J}\rho_{R_J}B$. We have  $(x,k)\notin \al_k[B_k]$, so $x\notin B_k$.  But    $B_k\in P_k^+$, so there exists a $y\in\rho_{R_k}B_k$ with not $xR_k y$. Hence not $(x,k)R_J(y,k)$.

Now we show that $(y,k)\in\rho_{R_J} B$. Any member of $B$ has the form $(z,j)\in\al_j[B_j]$ for some $j$ with $z\in B_j$.
If $j\ne k$, then $(z,j)R_J(y,k)$ by definition of $R_J$. But if $j=k$, then $z\in B_k$, hence $zR_k y$ as $y\in\rho_{R_k}B_k$, giving $(z,j)=(z,k)R_J(y,k)$ again. This proves that $(y,k)\in\rho_{R_J} B$. Since not $(x,k)R_J(y,k)$, we have $(x,k)\notin \lam_{R_J}\rho_{R_J}B$ as required to prove that $B$ is stable and complete the proof that $\thet$ is surjective.
\end{proof}

It is notable that the direct sum $\sum_JP_j$ and the bounded morphisms $\{\al_j,\be_j:j\in J\}$ form a \emph{coproduct} of 
 $\{P_j: j\in J\}$ in the category $\opol$. This means that for any $\Om$-polarity $P$ and bounded morphisms
 $\{\al'_j,\be'_j\colon P_j\to P : j \in J\}$, there is exactly one bounded morphism $\al,\be\colon \sum_JP_j\to P$ that factors each $\al'_j,\be'_j$ through $\al_j,\be_j$, i.e.\ $\al'_j,\be'_j=(\al,\be)\circ(\al_j,\be_j)$. The only maps that could do this are given by $\al(x,j)=\al'_j(x)$ and $\be(y,j)=\be'_j(y)$. It is left as an exercise for the reader to confirm that $\al,\be$ as thus defined is indeed a bounded morphism.
 
\section{Saturated extensions of $\Om$-polarities}

Recall that  we take the \emph{canonical extension} of an $\Om$-polarity $P$ to be the structure $(P^+)_+$.
Regarding $P$ as a model for a first-order language, we will now show that a sufficiently saturated elementary extension of $P$ can be mapped to the canonical extension $(P^+)_+$ by a bounded morphism.

Let $\sL=\{\ov X,\ov Y,\ov R,\ov S,\ov T\}$ be a signature  consisting of relation symbols corresponding to the different components of an  $\Om$-polarity. Fix an $\sL$-structure
$P=(X,Y,R,S,T)$ that is an $\Om$-polarity, and let $\sL_P=\{\ov A:A\in P^+\}$, where each $\ov A$ is a unary relation symbol. Then $P$ expands to an $\sL_P$-structure, which we continue to call $P$, by interpreting each symbol $\ov A$ as the set $A$.

Now let  
$$
P^*=(X^*,Y^*,R^*,S^*,T^*,\{A^*:A\in P^+\})
$$ 
be an $\sL_P$-structure that is an $\om$-saturated elementary extension of $P$. Then $P$ and $P^*$ satisfy the same first-order $\sL_P$-sentences. For each $A\in P^+$, $A^*$ is the subset of $X^*$ interpreting the symbol $\ov A$.  To explain $\om$-saturation, consider the process of taking a set $Z$ of elements of $P^*$, expanding $\sL_P$ to $\sL_P^Z$ by adding a set $\{\ov z:z\in Z\}$ of individual constants, and expanding $P^*$ to an  $\sL_P^Z$-structure $(P^*,Z)$ by interpreting each constant $\ov z$ as $z$. Then $\om$-saturation of $P^*$ means that for any finite set $Z$ of elements of $P^*$, and any set  $\Gamma$ of 
$\sL_P^Z$-formulas,  if each finite subset of $\Gamma$ is  satisfiable in $(P^*,Z)$ ,then $\Gamma$ is  satisfiable in 
$(P^*,Z)$.

The sentence
$\forall v_0\forall v_1\big(v_0\ov Rv_1\to\ov X(v_0)\land\ov Y(v_1)\big)$
is true in $P$, hence is true in $P^*$, implying that  $R^*\sub X^*\times Y^*$. Similarly we can show that
$S^*\sub  (X^*)^n\times Y^*$ and $T^*\sub  X^*\times (Y^*)^m$. 

As is common, we write $\ph(\vv v)$ to indicate that the list $\vv v$ of variables includes all the free variables of formula $\ph$. Then $\ph(\vv w)$ denotes the formula obtained by freely replacing each free occurrence of each $v_i$ in $\ph$ by $w_i$.

Now for a formula $\ph(\vv v,w)$ let  $\rho\ph(\vv v,w)$ be the formula
$$
\forall u(\ph(\vv v,u)\to u\ov Rw),
$$
and let  $\lam\ph(\vv v,w)$ be
$$
\forall u(\ph(\vv v,u)\to w\ov Ru),
$$
where $u$ is some fresh variable.
The sentence $\forall w(\lam\rho\ov A(w)\to\ov A(w))$ is true in $P$ when $A\in P^+$, since $A$ is stable. Hence the sentence is true in $P^*$, giving that $\lam_{R^*}\rho_{R^*}A^*\sub A^*$, showing that $A^*\in(P^*)^+$.

Now let $\ph(\vv v,w)$ be the atomic formula $\ov S(\vv v, w)$. The sentence 
$$
\forall \vv v\forall w( \lam\rho\ov S(\vv v,w) \to \ov S(\vv v,w))
$$ 
is true in $P$ and hence in $P^*$, giving that all sections of the form $S^*[\vv x,-]$ are stable in $P^*$.
Similar arguments establish the stability of  all other sections of $S^*$ and all sections of $T^*$.
Thus $P^*$ is an $\Om$-polarity.

We will construct a bounded morphism $\al,\be\colon P^*\to (P^+)_+$ from $P^*$ to 
$$
(P^+)_+=(\fF_{P^+},\fI_{P^+},\lap,S_{P^+},T_{P^+}),
$$
the canonical structure of $P^+$. The dual 
 $((P^+)_+)^+\to (P^*)^+$ of this bounded morphism,
given by Theorem \ref{homdual},  proves to be a lattice embedding of the canonical extension $(P^+)^\sg$ of $P^+$ into the stable set lattice of $P^*$. The construction of this bounded morphism $\al,\be$ follows the methodology used in \cite[Section 3.6]{gold:vari89} for the corresponding result for standard relational structures.

For $x\in X^*$, define $$\al(x)=\{A\in P^+:x\in A^*\}.$$ 
Then $\al(x)$ is non-empty, since it contains $X$.
For $A,B\in P^+$, the sentence
$$
\forall v\big(\,\ov{A\cap B} (v) \leftrightarrow \ov A(v) \land \ov B(v)\big)
$$
is true in $P$, hence in $P^*$, making $(A\cap B)^*=A^*\cap B^*$.
 So  $A\cap B\in\al(x)$ iff $A\in\al(x)$ and $B\in\al(x)$. Thus  $\al(x)$ is a filter of $P^+$, and so $\al$ maps $X^*$ into $\fF_{P^+}$. 

For $y\in Y^*$, 
 define $$\be(y)=\{A\in P^+: y\in (\rho_R A)^*\}.$$
 In $P$ we have $\rho_R\lam_R Y=Y$ and $\rho_R(A\lor B)=\rho_R A\cap \rho_R B$, the latter because
 $$
 \rho_R(A\lor B)=\rho_R\lam_R\rho_R(A\cup B)=\rho_R(A\cup B)=\rho_R A\cap \rho_R B.
 $$
 So the sentences $\forall v(\ov{Y}(v)\to\ov{\rho_R\lam_R Y}(v))$ and
 $$
 \forall v\big(\,\ov{ \rho_R(A\lor B)} (v) \leftrightarrow \ov{\rho_R A}(v) \land \ov{\rho_R B}(v)\big)
 $$
  are true in $P$, hence in $P^*$.
 Thus  any $y\in Y^*$ has $\lam_R Y\in \be(y)$, and  $A\lor B\in\be(y)$ iff $A\in\be(y)$ and $B\in\be(y)$, i.e.\  $\be(y)$ is an ideal of $P^+$. This shows that $\be$ maps $Y^*$ into $\fI_{P^+}$. 
 
 For each $A\in P^+$, we have
 \begin{equation}  \label{rhostar}
 (\rho_R A)^* = \rho_{R^*}A^*.
\end{equation}
This follows because the sentence
$
\forall w(\ov{\rho_R A}(w) \leftrightarrow \forall v(\ov A(v) \to v \ov R w)
$
is true in $P$, hence in $P^*$.

\begin{theorem}  \label{FineThm}
The pair $\al,\be$ is a bounded morphism from $ P^*$   to $(P^+)_+$.
\end{theorem}
\begin{proof}
To show that the map $\al$ is isotone, first define
 $v\le_1 w$ to  be an abbreviation for the formula
$
\forall u(v\ov R u\to w\ov R u).
$
This formula defines the relation $\le_1$ on $X$ determined by $R$ as in \eqref{leone}, 
and the corresponding relation $\le_1^*$ on $X^*$ determined by $R^*$.
But if $A\in P^+$ then $A$ is a $\le_1$-upset, so the sentence
$\forall v\forall w(\ov A(v)\land v\le_1 w\to\ov A(w))$
is true in $P$, hence in $P^*$, showing that $A^*$ is a $\le_1^*$-upset of $X^*$. Thus if $x,x'\in X^*$ have $x\le_1^* x'$, then $A\in\al(x)$ implies $A\in\al(x')$, showing that $\al(x)\sub\al(x')$, hence  by Lemma \ref{lesub}, $\al(x)\le_1\al(x')$ as required.

A similar argument shows that as $\rho_R A$ is stable, hence a $\le_2$-upset of $Y$, $(\rho_R A)^*$ is  a $\le_2^*$-upset of $Y^*$. Thus if $y\le_2^* y'$, then  $A\in\be(y)$ implies $A\in\be(y')$, so $\be(y)\sub\be(y')$. Hence $\be$ is isotone.

Our main task is to show that $\al,\be$ satisfy the back and forth conditions  of Definition \ref{defmorph}. For (1$_R$), suppose $\al(x)\lap\be(y)$. Then there is some $A\in\al(x)\cap\be(y)$, so $x\in A^*$ and $y\in (\rho_R A)^*$. 
Hence $y\in \rho_{R^*}A^*$ by \eqref{rhostar}, so  $xR^*y$ as required for (1$_R$).

For (2$_R$), take  $F\in\fF_{P^+}$ and $y\in Y^*$. We have to show that
$\al\inv[F)_1 R^*y$ implies $F\lap\be(y)$, where $[F)_1=\{F'\in\fF_{P^+}:F\sub F'\}$. We prove the contrapositive of this implication.

 Suppose that  $F\lap\be(y)$ fails, i.e.\ $F\cap\be(y)=\emptyset$. Consider the set of formulas
 $$
 \Gamma=\{\neg (v\ov R\ov y)\}\cup\{\ov A(v):A\in F\}
 $$
in the single variable $v$, where $\ov y$ is a constant denoting y. We show $\Gamma$ is finitely satisfiable in $(P^*,y)$. Given any finite $Z\sub F$, let $A=\bigcap Z\in F$. Then by assumption $A\notin\be(y)$, so $y\notin(\rho_R A)^*$.
Hence by \eqref{rhostar} there exists an $x\in A^*$ such that not $xR^* y$. Then
for all $B\in Z$, as $A\sub B$ we get $A^*\sub B^*$ by the truth of the sentence
$\forall w(\ov{A}(w)\to\ov{B}(w))$, so  $x\in B^*$.  Thus $x$ satisfies 
$\{\neg (v\ov R\ov y)\}\cup\{\ov{B}(v):B\in Z\}$.

This proves that $\Gamma$ is finitely satisfiable in $(P^*,y)$. By saturation it follows that $\Gamma$ itself is satisfiable in $P^*$ by some $x\in\bigcap\{A^*:A\in F\}$ with not $xR^* y$. Then $F\sub\al(x)$, so $x\in \al\inv[F)_1$, and hence not $\al\inv[F)_1R^*y$, giving (2$_R$).

The proof of (3$_R$) is similar: if $x\in X^*$, $D\in \fI_{P^+}$ and not $\al(x)\lap D$, let
$$
\Delta=\{\neg (\ov x\ov Rv)\}\cup\{\ov{\rho_R A}(v):A\in D\}.
$$
For any finite  $Z\sub D$, let $A=\join Z\in D$. Then as $\al(x)\cap D=\emptyset$, $A\notin\al(x)$ and hence $x\notin A^*$. Since $A^*$ is stable in $P^*$, 
there is some $y\in Y^*$ with $y\in \rho_{R^*}(A^*)$ and not $xR^*y$. Then for all $B\in Z$, as $B\sub A$ we get  $B^*\sub A^*$, hence  $y\in \rho_{R^*}(B^*)$.
Thus $y$ satisfies $\{\neg \big(\ov x\ov Rv\big)\}\cup\{\ov{\rho_R B}(v): B\in Z\}$ in $(P^*,x)$.

This proves that $\Delta$ is finitely satisfiable in in $(P^*,x)$. By saturation it follows that $\Delta$  is satisfiable in $(P^*,x)$ by some $y\in\bigcap\{(\rho_R A)^*:A\in D\}$ with not $xR^* y$. 
Then $D\sub\be(y)$, so $y\in \be\inv[D)_2$, and hence not $xR^*\be\inv[D)_2$, giving (3$_R$).

For (1$_S$), suppose $\al(\vv x)S_{P^+}\be (y)$. We want $\vv xS^*y$. We have some $\vv A\in_\pi\al(\vv x)$ with $f_S(\vv A)\in\be(y)$. Then $A_i\in\al(x_i)$, i.e.\ $x_i\in A_i^*$, for all $i<n$, while 
$y\in (\rho_\lap f_S(\vv A))^*=(\rho_\lap\lam_\lap f_S^\bullet(\vv A))^*= (f_S^\bullet(\vv A))^*$ as $f_S^\bullet(\vv A)$ is stable. Now the sentence
$$\textstyle
\forall\vv v\forall w\big[ \ov{f_S^\bullet(\vv A)}(w)  \land      \bigwedge_{i<n}\ov A(v_i)  \to \ov S(\vv v,w)\big]
$$
is true in $P$ by definition of $f_S^\bullet(\vv A)$, hence is true in $P^*$. This implies $\vv xS^*y$, as required for (1$_S$).

For (2$_S$), we must show that  $\al\inv[\vv F)_1 S^*y$ implies $\vv F S_{P^+} \be(y)$, where $\vv F\in(\fF_{P^+})^n$ and $y\in Y^*$. Suppose that  $\vv F S_{P^+} \be(y)$ fails. Then for all $\vv A\in_\pi\vv F$, $f_S(\vv A)\notin\be(y)$ and so
$y\notin  (\rho_\lap f_S(\vv A))^*=(f_S^\bullet(\vv A))^*$. Now let
$$
\Gamma'=\{\neg \ov S(\vv v,\ov y)\}\cup\{\ov A(v_0):A\in F_0\}\cup\cdots\cup \{\ov A(v_{n-1}):A\in F_{n-1}\}.
$$
 We show $\Gamma$ is finitely satisfiable in $(P^*,y)$. As each filter $F_i$ is closed under finite intersections, it is enough to show that if $A_i\in F_i$ for all $i<n$, then the set
 $$
 \Gamma'_0=\{\neg \ov S(\vv v,\ov y)\}\cup\{\ov A_0(v_0),\dots,\ov A_{n-1}(v_{n-1})\}
 $$
 of formulas in the free variables $\vc{v}{n}$
 is satisfiable. For such $A_i$ we have $\vv A=(\vc{A}{n})\in_\pi\vv F$, so $y\notin (f_S^\bullet(\vv A))^*$ by the above.
But the sentence
$$\textstyle
\forall w\big[ \,\ov Y(w)\land \neg\ov{f_S^\bullet(\vv A)}(w) \to\exists\vv v \big(\bigwedge_{i<n}\ov A(v_i)\land\neg\ov S(\vv v,w)\big) \big]
$$
is true in $P$, hence in $P^*$, so we infer that there exist $x_i\in A_i^*$ for all $i<n$ such that not $S^*(\vv x,y)$.
Thus $\vv x$ satisfies $\Gamma_0'$ in in $(P^*,y)$.

By saturation it follows that $\Gamma'$ is satisfied by some $n$-tuple $\vv x$. Then  for all $i<n$ we  have 
$x_i\in\bigcap\{A^*:A\in F_i\}$, so $F_i\sub\al(x_i)$. Thus $\vv x\in \al\inv[\vv F)_1 $. But $\vv xS^* y$ fails, therefore so does
$\al\inv[\vv F)_1 S^*y$. This completes the proof of (2$_S$). 

The cases of (1$_T$) and  (2$_T$) are similar to the above.  For (1$_T$), suppose $\al(x)T_{P^+}\be(\vv y)$. 
We want $ xT^*\vv y$. We have
some $\vv A\in_\pi\be(\vv y)$ with $g_T(\vv  A)\in\al(x)$.  
Then $A_i\in\be(y_i)$, i.e.\ $y_i\in (\rho_R A_i)^*$, for all $i<m$, while $x\in (g_T(\vv  A))^*$.
 Now the sentence
$$\textstyle
\forall v\forall \vv w\big[ \, \ov{g_T(\vv A)}(v)  \land      \bigwedge_{i<m}\ov{\rho_R A_i}(w_i)  \to \ov T( v,\vv w)\big]
$$
is true in $P$ by definition of $g_T(\vv  A)$, hence is true in $P^*$. This implies $ xT^*\vv y$, as required for (1$_T$).

For (2$_T$), suppose that not $\al(x)T_{P^+}\vv D$, where $x\in X^*$ and $\vv D\in(\fI_{P^+})^m$. We show that not
$xT^*\be\inv[\vv D)_2$. We have for all $\vv A\in_\pi\vv D$ that $g_T(\vv A)\notin\al(x)$ and so $x\notin (g_T(\vv A))^*$.
Now let
$$
\Delta'=\{\neg\ov T(\ov x,\vv w)\}\cup \{\ov{\rho_RA}(w_0):A\in D_0\}\cup\cdots\cup \{\ov{\rho_RA}(w_{m-1}):A\in D_{m-1}\}.
$$
Take any finite sets $\vc{E}{m}$ such that $E_i\sub D_i$ for all $i<m$ We show that the finite set
$$
\Delta'_0=\{\neg\ov T(\ov x,\vv w)\}\cup \{\ov{\rho_RB}(w_0):B\in E_0\}\cup\cdots\cup \{\ov{\rho_RB}(w_{m-1}):B\in E_{m-1}\}
$$
 of formulas in the free variables $\vc{w}{m}$
is satisfiable in $(P^*,x)$. For each $i<m$, let $A_i=\join E_i\in D_i$. Then $\vv A=(\vc{A}{m-1})\in_\pi\vv D$, and so 
 $x\notin (g_T(\vv A))^*$ by the above.
But the sentence
$$\textstyle
\forall v\big[ \,\ov X(v)\land \neg\ov{g_T(\vv A)}(v) \to\exists\vv w \big(\bigwedge_{i<m}\ov{\rho_RA_i}(w_i)\land\neg\ov T( v,\vv w)\big) \big]
$$
is true in $P$, hence in $P^*$,  so we infer that there exist $y_i\in (\rho_RA_i)^*$ for all $i<m$ such that not $T^*( x,\vv y)$.
Then for each $i$, any $B\in E_i$ has $B\sub A_i$, so $\rho_R A_i\sub\rho_R B$, hence $(\rho_R A_i)^*\sub(\rho_R B)^*$ and thus $y_i\in (\rho_RB)^*$.
Thus $\vv y$ satisfies $\Delta'_0$ in $(P^*,x)$.

This proves that $\Delta'$ is finitely satisfiable in $P^*$. By saturation it follows that $\Delta'$ is satisfied in $P^*$ by some $m$-tuple $\vv y$.  Then  for all $i<m$ we  have 
$y_i\in\bigcap\{(\rho_RA)^*:A\in D_i\}$, so $D_i\sub\be(y_i)$. Thus $\vv y\in \be\inv[\vv D)_2 $. But $ xT^*\vv y$ fails, therefore so does
$xT^*\be\inv[\vv D)_2$. This completes the proof of (2$_T$). 
\end{proof}

\begin{example} \label{nonsurj}
The bounded morphism of Theorem \ref{FineThm} does not in general have surjective first component, in spite of the saturation of $P^*$. To see this, consider a $P$ having the property  $\lam_R Y=\emptyset$, i.e.\ for all $x\in X$ there exists $y\in Y$ such that not $xRy$. This is a first-order condition, so it is preserved by elementary extensions. It is also preserved by  images of bounded morphisms with surjective first component. For, if $\al,\be:P\to P'$ is any bounded morphism with surjective $\al$, and $\lam_{R'} Y' \ne \emptyset$, then by the surjectivity there must exist an element of $\al\inv\lam_{R'} Y' $, which equal to  $\lam_R(\be\inv Y')$, so $\lam_R Y\ne\emptyset$.

Now there are many $P$ satisfying  $\lam_R Y=\emptyset$ (e.g.\  any  with polarity of the form $(X,X,\ne)$). For such $P$, if there was \emph{any} bounded morphism $P^*\to (P^+)_+$ with surjective $\al$, then $(P^+)_+$ would satisfy the preserved condition. But that is not so, as  no canonical structure $\LL_+$   has $\lam_\lap\fI_\LL=\emptyset$. This is because 
$\fF_\LL$ contains the  filter $\LL$ which intersects every ideal of $\LL$, so $\LL\in\lam_\lap\fI_\LL$.

Thus if $P$ has $\lam_R Y=\emptyset$, then there is no bounded morphism $P^*\to (P^+)_+$ with surjective $\al$, where $P^*$ is any elementary extension of $P$.
\end{example}

\section{Maximal covering morphisms}

In the case of Kripke frames, for Boolean modal logics or distributive substructural logics, the corresponding version of Theorem \ref{FineThm} produces a bounded morphism that \emph{is} surjective. The proof depends on the points of $(P^+)_+$ being \emph{prime} filters  \cite[3.6]{gold:vari89}.
But here, in dealing with possibly non-distributive lattices, we  admit arbitrary filters as points in canonical structures.

In the lattice representations developed by Urquhart \cite{urqu:topo78} and Hartung \cite{hart:topo92}, the points of representing spaces are filters or ideals, or filter-ideal pairs, that have certain mutual maximality properties. As these papers point out, this does not lead to a good duality construction for lattice homomorphisms, because the preimages of maximal filters under lattice homomorphisms need not be maximal. Here we have seen in Theorem \ref{morphdual} that admitting arbitrary filters leads to a notion of dual morphism for lattice homomorphisms that has good properties.

Now surjective bounded morphisms are logically important because  they preserve validity of formulas in a semantics based on the structures involved. Typically, the validity of a formula in $P$ will be equivalent to the satisfaction of some corresponding equation by the algebra $P^+$. So the preservation of formula validity in passing from $P$ to $P'$ will be secured if equation satisfaction is preserved in passing from $P^+$ to $(P')^+$. This will hold if $(P')^+$ is isomorphic to a subalgebra of $P^+$. That will in turn hold if the dual $(\al,\be)^+$ of some bounded morphism $\al,\be:P\to P'$ is injective. For this it suffices, by Theorem \ref{homdual}, that $\al$ be surjective.
But when $P'$ is a canonical structure, we can ensure this injectivity of $(\al,\be)^+$ by a condition that is weaker than the surjectivity of $\al$.

To define this condition, first define a filter $F$ of a lattice $\LL$ to be \emph{$D$-maximal}, where $D$ is an ideal of $\LL$, if $F$ is maximal in the set of all filters disjoint from $D$.
Call $F$ \emph{i-maximal} if it is $D$-maximal for some ideal $D$.
Now define a \emph{maximal covering  morphism} to be a bounded morphism $\al,\be:P\to\LL_+$ into some canonical structure such that  the image $\al[X]$ of the map 
$\al:X\to\fF_\LL$ includes all i-maximal filters of $\LL$. This  condition holds immediately if $\al$ is surjective. 

\begin{theorem}  \label{presinj}
If  $\al,\be:P\to\LL_+$ is a maximal covering   morphism, then the $\Om$-lattice  homomorphism 
 $(\al,\be)^+\colon (\LL_+)^+\to P^+$ is injective.
\end{theorem}
\begin{proof}
 $(\al,\be)^+$ is the map $A\mapsto \al\inv A$.
 Suppose that $A$ and $B$ are stable subsets of $\fF_{\LL}$ in $\LL_+$, with $A\ne B$. Then, say, $A\nsubseteq B$, and there is some $F\in A$ with $F\notin B$. By stability of $B$ there is some $D\in \rho_\lap B$ such that not $F\lap D$, hence $F\cap D=\emptyset$. By Zorn's Lemma, $F$ can be extended to an $F'\in\fF_{\LL}$ that is $D$-maximal. By Lemmas \ref{upsets} and \ref{lesub}, $A$ is a $\sub$-upset, so we get $F'\in A$. Since $D\in \rho_\lap B$ and $F'\cap D=\emptyset$, we have
 $F'\notin\lam_\lap\rho_\lap B=B$.
  Since $F'$ is i-maximal and $\al,\be$ is maximal covering , there exists an $x\in X$ such that $\al(x)=F'$. So $x\in\al\inv A\setminus\al\inv B$. Hence $ \al\inv A\ne\al\inv B$. Thus  $(\al,\be)^+$ is injective.
\end{proof}

\begin{theorem}  \label{preserve}
The bounded morphism $\al,\be\colon P^*\to (P^+)_+ $ of Theorem \ref{FineThm} is maximal covering.
Hence there is  an $\Om$-lattice  monomorphism
$(P^+)^\sg\rightarrowtail (P^*)^+$.
\end{theorem}
\begin{proof}
Let $F\in\fF_{P^+}$ be i-maximal. Then there is some ideal $D\in\fI_{P^+}$ such that $F$ is $D$-maximal.
Now let
 $$
 \Gamma = \{\ov A(v):A\in F\}\cup\{\neg\ov B(v):B\in D\}.
 $$
 We show that $\Gamma$ is finitely satisfiable in $P^*$. Given finite sets
 $G\sub F$ and $E\sub D$, put
 $A=\bigcap G\in F$ and $B= \join E\in D$. If $A\sub B$, then $B\in F$ (and $A\in D$), contradicting $F\cap D=\emptyset$. So there must be some $x\in X$ with $x\in A\setminus B$. Hence 
 $x\in A'\setminus B'$ for all $A'\in G$ and $B'\in E$.
 So $x$ satisfies $ \{\ov{A'}(v):A'\in G\}\cup\{ \neg\ov{B'}(v):B'\in E\}$ in $P$.
 
This shows that each finite subset of $\Gamma$ is satisfiable in $P$, hence is satisfiable in its elementary extension $P^*$.
By saturation it follows that $\Gamma$  is satisfiable in $P^*$ by some $x\in\bigcap\{ A^*:A\in F\}$, with $x\notin B^*$ for all $B\in D$. Thus the filter $\al(x)$ includes $F$ and is disjoint from $D$. The maximality of $F$ in this respect implies that 
$\al(x)=F$. Thus the image of $\al$ includes all i-maximal filters of $P^+$, as required for the morphism to be maximal covering.

It follows from Theorem \ref{presinj} that the  $\Om$-lattice homomorphism  
$$
(\al,\be)^+\colon ((P^+)_+)^+\to (P^*)^+
$$
given by Theorem \ref{homdual}  is injective. But $((P^+)_+)^+=(P^+)^\sg$.
\end{proof}

We can strengthen this construction as follows.

\begin{corollary}  \label{maxboundsurj}
Let $\al_1,\be_1 \colon P\to P_1$ be a bounded morphism with $\al_1\colon X\to X_1$ being surjective. Then there is a maximal covering  morphism from  $P^*$ to $(P_1^+)_+$.
\end{corollary}

\begin{proof}
We have the situation
$$
\xymatrix{ {P^*}   \ar[r]^-{\al,\be}  & { (P^+)_+ } \ar[r]^-{\al_2,\be_2} &  (P_1^+)_+},
$$
where $\al,\be$ is the maximal covering  morphism of Theorem \ref{preserve},
and  $\al_2,\be_2 $ is the double dual of $\al_1,\be_1$, i.e.\ $\al_2=(\al_{1}^+)_+$
and $\be_2=(\be_{1}^+)_+$.
In particular, $\al_2(G)=\{A\in P_1^+:\al_1\inv A\in G\}$ for any filter $G$ of $P^+$.
The composite pair $\al_2\circ\al,\be_2\circ\be$ is a bounded morphism  $P^* \to (P_1^+)_+$, so it suffices to show that it is maximal covering .
We adapt and extend the argument of Theorem \ref{preserve}. Let
$F\in\fF_{P_1^+}$ be D-maximal, where $D\in\fI_{P_1^+}$.
Put
 $$
 \Gamma = \{\ov{\al_1\inv A}(v):A\in F\}\cup\{\neg\ov{\al_1\inv B}(v):B\in D\}.
 $$
For any finite sets
 $G\sub F$ and $E\sub D$, as in the proof of Theorem \ref{preserve} there is  some $z\in X_1$ with  $z\in A\setminus B$ for all $A\in G$ and $B\in E$. As $\al_1$ is surjective, there exists $x\in X$ with $\al_1(x)=z$, so
  $x\in \al_1\inv A\setminus \al_1\inv B$ for all $A\in G$ and $B\in E$.
 So $x$ satisfies $ \{\ov{\al_1\inv A}(v):A\in G\}\cup\{\neg\ov{\al_1\inv B}(v):B\in E\}$ in $P$, hence in $P^*$.
 
 By saturation it follows that $X^*$ has a point $x$ that satisfies $\Gamma$, so belongs to $(\al_1\inv A)^*$ for all $A\in F$ but not to
 $(\al_1\inv B)^*$ for any $B\in D$. Thus $\al(x)$ includes  $\{\al_1\inv A:A\in F\}$ and is disjoint from
$\{\al_1\inv B:B\in D\}$.
But  $\al_2(\al(x))=\{A\in P_1^+:\al_1\inv A\in \al(x)\}$, which includes $F$ and is disjoint from $D$. Hence
 $\al_2(\al(x))=F$ by $D$-maximality of $F$.
 
 This proves that the image of $\al_2\circ \al$ includes all i-maximal filters of $P_1^+$, as required.
\end{proof}

There is a further significant corollary to Theorem \ref{preserve}, which follows from the fact that saturated models can be obtained as ultrapowers. By the theory of \cite[Section 6.1]{chan:mode73}, for any $P$ there is an ultrafilter $U$ such that $P^U$ is $\om$-saturated, where $P^U$ is the ultrapower of $P$ modulo $U$. We take $P^U$ as $P^*$ in Theorem \ref{preserve} and Corollary \ref{maxboundsurj}:

\begin{corollary}  \label{A3}
For any\/  $\Om$-polarity $P$ and bounded morphism $\al_1,\be_1 \colon P\to P_1$ with $\al_1$ surjective,
there is an ultrafilter $U$ such that there is a maximal covering  morphism $P^U\to (P_1^+)_+$  and 
an\/ $\Om$-lattice  monomorphism
$(P_1^+)^\sg\rightarrowtail (P^U)^+$.
\qed
\end{corollary}

In \cite{gold:defi18}, this corollary with $P_1=P$ was obtained for a polarity $P$ for which $P^+$ has additional $\Om$-indexed complete normal operators and dual operators  that are all first-order definable over $P$ in the sense indicated in \eqref{firstfsA} and  \eqref{firstgT} for
the operations $f_S$ and $g_T$.
The method used in \cite{gold:defi18} was focused  on the algebraic side of the duality between algebras and structures.  It showed that $(P^U)^+$ is a MacNeille completion of the ultrapower algebra $(P^+)^U$ and then appealled to a result from 
\cite{gehr:macn06} stating that there is an embedding of $(P^+)^\sg$ into the MacNeille completion of $(P^+)^U$ for a suitable $U$ that has $(P^+)^U$ sufficiently saturated.

Corollary \ref{A3} gives one of the properties that define the notion of a \emph{canonicity framework}, as introduced in \cite{gold:cano18}. Such a framework describes a set of relationships between a class $\Sigma$ of structures and a \emph{variety}, 
 i.e.\ equationally definable class, $\CC$ of algebras equipped with operations  $(-)^\sg\colon\CC\to\CC$ and $(-)^+\colon\Sigma\to\CC$, that are sufficient to ensure that the following holds.
\begin{itemize}
\item[(\ddag)]
if $\cS$ is any subclass of $\Sigma$ that is closed under ultraproducts, then
the variety of algebras generated by  $\cS^+=\{P^+:P\in\cS\}$ is closed under the  operation $(-)^\sg$.
\end{itemize}
This provides  an axiomatic formulation of a result about the generation of varieties closed under canonical extensions that was first proven in \cite{gold:vari89} for Boolean algebras with operators, and which was itself an algebraic generalisation of a theorem of Fine \cite{fine:conn75} stating that a first-order definable class of Kripke frames characterises a modal logic that is valid in its canonical frames.

A canonicity framework can be formed  by taking $\Sigma$ to be the class of $\Om$-polarities, $\CC$ to be the variety of 
$\onlo$'s, $\LL^\sg$ to be the canonical extension \eqref{Lsigom}, and $P^+$ the stable set lattice \eqref{PplusOm}.
Hence the conclusion of (\ddag) holds if $\cS$ is any class of $\Om$-polarities that is closed under ultraproducts.

\section{Goldblatt-Thomason theorem}

This theorem \cite[Theorem 8]{gold:axio75} was originally formulated  as an answer to the question: which first-order definable properties of a binary relation can be expressed by modal axioms? It gave structural closure conditions on a first-order definable class of Kripke frames that are necessary and sufficient for that class to be the class of all frames that validate the theorems of some propositional modal logic. In this section we will derive two results of this kind in the present setting of polarity-based structures.

Now a Kripke frame validates a particular modal formula iff the dual algebra of the frame satisfies some modal algebraic equation \cite[Prop.~5.24]{blac:moda01}. Hence there is a correspondence between modal logics and varieties  of modal algebras \cite[Theorem~5.27]{blac:moda01}. Accordingly, the kind of result we seek is one that characterises when a class of $\Om$-polarities is the class of all such structures whose dual algebras $P^+$ belong to some variety.
If $\V$ is a variety of $\Om$-lattices, let
$$
\cS_\V=\{P:P^+\in\V\}
$$
be the class of all  $\Om$-polarities whose  stable set lattice belongs to $\V$.  We will  give structural closure conditions on  a class $\cS$ of structures that characterise when it is of the form $\cS_\V$.

We say that $\cS$ is \emph{closed under direct sums} if, whenever  $\{P_j:j\in J\}\sub \cS$, then $\sum_J P_j\in\cS$.
$\cS$ is \emph{closed under inner substructures} if, whenever $P'\in\cS$ and  $P$ is an inner substructure of $P'$, then  $P\in\cS$.
$\cS$ is \emph{closed under images  of surjective morphisms} when, for any bounded morphism $\al,\be\colon P\to P'$ with $\al$ surjective, if  $P\in\cS$, then  $P'\in\cS$.
$\cS$ is \emph{closed under codomains of maximal covering  morphisms} if, whenever  there is an maximal covering  morphism from $P$ to $\LL_+$ and  $P\in\cS$, then  $\LL_+\in\cS$. 
  $\cS$  \emph{reflects canonical extensions} if $P\in\cS$ whenever
$(P^+)_+\in\cS$ (equivalently, if the complement of $\cS$ is closed under canonical extensions).

\begin{lemma} \label{maxtosurj}
Suppose that  $\cS$ reflects canonical extensions.
\begin{enumerate}[\rm(1)]
\item
If $\cS$ is closed under canonical extensions and  codomains of  maximal covering   morphisms, then it is closed under images of  surjective  morphisms.
\item
If $\cS$ is closed under ultrapowers and  codomains of  maximal covering   morphisms, then it is closed under images of  surjective  morphisms.
\end{enumerate}
\end{lemma}

\begin{proof}
Let $P\in\cS$, and suppose there is a  bounded morphism $\al_1,\be_1 \colon P\to P_1$ with $\al $  surjective. 

(1):  By Theorems \ref{homdual} and \ref{morphdual}, $(\al_1^+)_+,(\be_1^+)_+\colon (P^+)_+\to (P_1^+)_+$ is a bounded morphism, and is maximal covering  as $(\al_1^+)_+$ is surjective. Thus if $\cS$ is closed under canonical extensions and  codomains of  maximal covering   morphisms, then $P\in\cS$ implies $(P_1^+)_+\in \cS$, hence $P_1\in\cS$ as $\cS$ reflects canonical extensions.

(2)
Taking $P^*$ to be an $\om$-saturated ultrapower of $P$, by Corollary \ref{maxboundsurj} there is a maximal covering  morphism from  $P^*$ to $(P_1^+)_+$.  Thus if $\cS$ is closed  under ultrapowers and  codomains of  maximal covering   morphisms, then $P\in\cS$ implies $(P_1^+)_+\in \cS$, hence $P_1\in\cS$.
\end{proof}

\begin{theorem}\label{cSVclosed}
For any variety $\V$, the class $\cS_\V$ reflects canonical extensions and is closed under direct sums, inner substructures, codomains of  maximal covering  morphisms, and images of  surjective  morphisms.
\end{theorem}
\begin{proof}
We use the fact that $\V$ is closed under direct products, subalgebras, and homomorphic images (including isomorphic images).

Reflection of canonical extensions: suppose $(P^+)_+\in\cS_\V$, so $(P^+)^\sg= ((P^+)_+)^+\in\V$. But $P^+$ is isomorphic to a subalgebra of $(P^+)^\sg$, so then $P^+\in\V$, hence $P\in\cS_\V$.

Closure under direct sums: if $\{P_j:j\in J\}\sub \cS_\V$, then  $\{P_j^+:i\in J\}\sub \V$, so by closure of $\V$ under products and isomorphism and Theorem \ref{sumprod} we get $(\sum_JP_j )^+ \in\V$, so $\sum_JP_j\in\cS_\V$.

Closure under   inner substructures: suppose $P$ is  an inner substructure of $P'\in\cS_\V$. By Theorem \ref{innersubdual} there is a surjective homomorphism $(P')^+\to P^+$. Since $(P')^+\in\V$ this implies $P^+\in\V$, hence $P\in\cS_\V$.

Closure under images  of surjective morphisms: this is dual to the previous case, using the result of Theorem \ref{homdual} that if morphism $\al,\be\colon P\to P'$ has $\al$  surjective, then it induces an injective homomorphism $(P')^+\to P^+$.
Hence $P^+\in\V$ implies $(P')^+\in\V$

Closure under   codomains of maximal covering  morphisms: suppose there is a maximal covering  morphism from $P$ to $\LL_+$  with  $P\in\cS_\V$. Then by Theorem \ref{presinj} there is an injective homomorphism making $(\LL_+)^+$ isomorphic to a subalgebra of $P^+\in\V$. Hence  $(\LL_+)^+\in\V$ and so $\LL_+\in\cS_\V$.
\end{proof}

Our first definability result is this:

\begin{theorem} \label{GT}
Let $\cS$ be closed under canonical extensions. Then the following are equivalent.
\begin{enumerate}[\rm(1)]
\item 
$\cS$ is equal to $\cS_\V$ for some variety $\V$.
\item
 $\cS$ reflects canonical extensions and is closed under direct sums, inner substructures  and codomains of  maximal covering   morphisms.
\item
 $\cS$ reflects canonical extensions and is closed under direct sums, inner substructures  and images of  surjective  morphisms.
\end{enumerate}
\end{theorem}
\begin{proof}
(1) implies (2): By Theorem \ref{cSVclosed}.

(2) implies (3): Assume (2).  Then in particular $\cS$   reflects canonical extensions and is closed under  canonical extensions  and codomains of  maximal covering  morphisms. These imply that $\cS$ is closed under images of  surjective  morphisms by  Lemma \ref{maxtosurj}(1). Hence (3) holds.

(3) implies (1):   Assume (3). Then we show that $\cS=\cS_\V$ where $\V$ is the variety generated by 
$\cS^+=\{P^+:P\in\cS\}$, i.e.\ the smallest variety that includes $\cS^+$. It is immediate that
 $\cS\sub\cS_\V$.
Conversely, suppose $P\in \cS_{\V}$. Then $P^+\in\V$, so from the well known analysis of the generation  of varieties,
$P^+$ is a homomorphic image of some algebra $\LL$ which is isomorphic to a subalgebra of a direct product $\prod_J P_j^+$ with $\{P_j:j\in J\}\sub\cS$.
But   $\prod_J P_j^+ \cong (\sum_JP_j )^+$ (Theorem \ref{sumprod}), and thus there are homomorphisms $\thet$ and $\chi$ having the configuration
\newdir{ >}{{}*!/-8pt/@{>}}
$$
\xymatrix{{P^+}   & {\LL} \ar@{->>}[l]_{\ \ \thet}  \ar@{ >->}[r]^-{\chi} &  (\sum_JP_j )^+},
$$
with $\thet$ surjective and $\chi$ injective. By Theorem \ref{morphdual}, there exist bounded morphisms
$$
\xymatrix{{(P^+)_+}   \ar@{ >->}[rr]^{\ \ \al_\thet,\be_\thet}  && {\LL_+}  &&\ar@{->>}[ll]_-{\ \al_\chi,\be_\chi}  ((\sum_JP_j )^+)_+},
$$
with $\al_\thet$ and $\be_\thet$ injective  and $\al_\chi$ and $\be_\chi$ surjective.
But $((\sum_JP_j )^+)_+\in\cS$, by closure under direct sums  and canonical extensions. Hence $\LL_+\in\cS$ by closure under images of surjective morphisms.
But  Theorem \ref{morphdual} also gives that $\al_\thet,\be_\thet$ makes  $(P^+)_+$ isomorphic to an inner substructure of 
$\LL_+$. By closure of $\cS$ under inner substructures and  images of isomorphisms (as a special case of images of surjective morphisms), this implies that   $(P^+)_+\in\cS$. Finally then  $\cS$ contains $P$ as it reflects canonical extensions. Thus $\cS=\cS_{\V}$ as required for (1).
\end{proof}

Now the equivalence of (1) and (3) of this theorem for Kripke frames in \cite[Theorem 8]{gold:axio75} has the hypothesis that $\cS$ is closed under first-order equivalence. Together with closure under images of bounded morphisms, this implies that $\cS$ is closed under canonical extensions, and the  proof of the main theorem  proceeds from there. Thus, although closure under first-order equivalence is already weaker than being first-order definable, the theorem for Kripke frames can be stated with the still weaker hypothesis of closure under canonical extensions, as has been above here for $\Om$-polarities. In \cite{gold:axio75} the closure under elementary equivalence was used to show that a Kripke frame $\mathcal F$ has a saturated elementary extension 
$\mathcal F^*$ that is  mapped by a surjective bounded morphism to $(\mathcal F^+)_+$. This $\mathcal F^*$ can be taken to be an ultrapower of $\mathcal F$, so an alternative hypothesis is that $\cS$ is closed under ultrapowers. In the present situation with polarities we do not get a surjection to $(P^+)_+$, but rather a maximal covering  morphism as in Theorem \ref{preserve}. But we can apply Lemma \ref{maxtosurj} to  Theorem \ref{GT}  to give the following definability characterisation.

\begin{theorem}  \label{GT2}
Let $\cS$ be a class of\/ $\Om$-polarities that is closed under ultrapowers. 
Then the following are equivalent.
\begin{enumerate}[\rm(1)]
\item 
$\cS$ is equal to $\cS_\V$ for some variety $\V$.
\item
 $\cS$ reflects canonical extensions and is closed under direct sums, inner substructures  and codomains of  maximal covering   morphisms.
\end{enumerate}
\end{theorem}

\begin{proof}
Let $\cS$ be closed under ultrapowers.
(1) implies (2) again by Theorem \ref{cSVclosed}. Conversely, assume (2). 
Then the closure of $\cS$  under ultrapowers and codomains of maximal covering morphisms ensures that $\cS$ is closed under  images of  surjective  morphisms by Lemma \ref{maxtosurj}(2), so (3) of Theorem \ref{GT} holds.
But it also ensures that $\cS$ is closed under canonical extensions, by Corollary \ref{A3}, which implies (with $P_1=P$)  that there is a maximal covering morphism from an ultrapower of $P$ to $(P^+)_+$. Hence (1) holds by Theorem \ref{GT}.
\end{proof}

There is a good reason why part (3) of Theorem \ref{GT} is not part of this result. Although the equivalence of parts (1) and (3)  holds for ultrapower-closed classes of modal Krikpe frames, it fails to hold in general for ultrapower-closed classes of polarities. For some such classes, (3) is strictly weaker than (1) and (2).
Thus the replacement of images of bounded morphisms by codomains of  maximal covering morphisms is essential here.

An example of this failure is the class $\cS_0$ of all polarities that satisfy $\lam_R Y=\emptyset$. It was observed in Example \ref{nonsurj} that this is a first-order definable condition. Hence  $\cS_0$ is closed under ultrapowers. It was also noted that $\cS_0$ is closed under images of  surjective  morphisms. It can be readily checked that it is  closed under  direct sums and inner substructures as well. Moreover it reflects canonical extensions, vacuously, because it contains no canonical structures $\LL_+$, hence none of the form $(P^+)_+$, as Example  \ref{nonsurj} explained. 
Thus $\cS_0$ fulfills part (3). However, since it is non-empty,  it is not closed under canonical extensions. Hence by Corollary \ref{A3} it is not closed under codomains of  maximal covering  morphisms, so it fails to satisfy part (2), and thus fails (1) as well.

\section{Further Studies}

We conclude by pointing out two possible directions for further study of  morphisms of polarities. One concerns the topological representation of lattices, imposing topologies on the sets $X$ and $Y$ in order to define a category
of topological $\Om$-polarities and continuous bounded morphisms that is \emph{dually equivalent} to $\olat$. This would involve functorial mappings $\A\mapsto \A_+$ and   $P\mapsto P^+$ such that 
$\A$ is naturally isomorphic to $ (\A_+)^+$ and   $P$ is naturally isomorphic to $(P^+)_+$. 
Guidance on how to go about topologising  can be found in such papers as 
\cite{prie:repr70,gold:ston75,urqu:topo78,gold:vari89,hart:topo92,hart:exte93,hart:ston97,gehr:dist14}.

The other development is  to generalise from operators to \emph{quasi}operators, operations that in each coordinate  either preserve joins or map meets to joins \cite{gehr:cano07,gehr:dual07,conr:algo16,conr:gold18}. For instance, any `negation' operation $\neg$ satisfying the De Morgan law $\neg(a\land b)=\neg a\lor\neg b$ is a unary quasioperator.
A \emph{dual quasioperator} is an operation that in each coordinate  either preserves meets or map joins to meets.  The negation operation of a Heyting algebra is a dual quasioperator that is not in general a quasioperator.

Operations of these types can be characterised by using the \emph{order dual} $\LL^\du$ of a lattice $\LL$. The partial order of $\LL^\du$ is the inverse of that of $\LL$, so the join and meet in 
$\LL^\du$ of a set of elements are the meet and join, respectively, of the same set in $\LL$.
An \emph{$n$-ary monotonicity type} is an $n$-tuple $\ep\in\{1,\du\}^n$ whose terms will be denoted  $\ep(i)$  for $i<n$. Putting $\LL^1=\LL$, we can then define $\LL^\ep$ to be the direct product lattice $\prod_{i<n}\LL^{\ep(i)}$.
A function  with domain $\LL^n$ can also be viewed as a function on $\LL^\ep$, and  
$f:\LL^n\to\LL$ is  an \emph{$\ep$-operator} if $f:\LL^\ep\to\LL$ is an operator.  An $n$-ary $f$ is a \emph{quasioperator} if it is an $\ep$-operator for some $\ep\in\{1,\du\}^n$.

Given a polarity $P=(X,Y,R)$, let $X^\ep$ be the product set $\prod_{i<n} X_i$, where $X_i$ is $X$ if $\ep(i)=1$ and is $Y$ if $\ep(i)=\du$.   Then $X^\ep$ is quasi-ordered by the product relation $\leq_\ep$, where 
$\vv z \leq_\ep \vv{z'}$ iff $z_i\leq^{\ep(i)}z'_i$ for all $i<n$, and $\leq^{\ep(i)}$ is $\leq_1$ when $\ep(i)=1$ and is  $\leq_2$ when $\ep(i)=\du$.
This yields upsets of the form
$[\vv z)_\ep=\{\vv{w}\in X^\ep:\vv z \leq_\ep \vv{w}\}$.

An $\ep$-operator $f_S$ on $P^+$ can be defined from a relation $S\sub X^\ep\times Y$. For $\vv A\in (P^+)^n$, let $\vv A^\ep=\Big(A_0^{\ep(0)},\dots,A_{n-1}^{\ep(n-1)}\Big)$ where $A_{i}^{\ep(i)}$ is $A_i$ if $\ep(i)=1$ and is $\rho_R A_i$ if $\ep(i)=\du$.  Then    $\pi\vv A^\ep\sub X^\ep$, and  we put
$$
f_S\vv A=\lam_R \{y\in Y: (\pi\vv A^\ep)Sy\}.
$$
Let $P'$ be a second polarity with a relation $S'\sub (X')^\ep\times Y'$, and let $\al:X\to X'$ and $\be:Y\to Y'$ be isotone maps that satisfy  (1$_R$)--(3$_R$). Define $\al_\ep:X^\ep\to(X')^\ep$ by putting 
$\al_\ep(\vv z)=\vv w$, where $w_i$ is $\al(z_i)$ if $\ep(i)=1$ and is $\be(z_i)$ if $\ep(i)=\du$.
The back and forth conditions to make $\al,\be$ a bounded morphism are then these:

\begin{enumerate}
\item[(1$_S$)]
$\al_\ep(\vv z)S'\be  (y)$ implies $\vv zSy$, \quad all $\vv z\in X^\ep,\  y\in Y$.
\item[(2$_S$)]
$(\al \inv[\vv{w})_\ep)Sy$ implies $\vv{w}S'\be (y)$, \quad all $\vv{w}\in (X')^\ep,\ y\in Y$.
\end{enumerate}

The constructions and notation have become more intricate, but there appears to be no impediment to carrying through the same analysis for quasioperators that we completed for operators, and to adapting it to dual quasioperators,  which can be constructed on a stable set lattice from  relations of the form $T\sub X\times Y^\ep$. Details are left to the  interested reader.

\bibliographystyle{plain}


\end{document}